\documentclass[11pt, a4paper]{amsart}

\usepackage{amsmath, amsthm, amssymb}

\newcommand{\bR}{\mathbb{R}}
\newcommand{\bC}{\mathbb{C}}
\newcommand{\bP}{\mathbb{P}}
\newcommand{\sP}{\mathsf{P}}
\newcommand{\sH}{\mathsf{H}}
\newcommand{\bN}{\mathbb{N}}
\newcommand{\rd}{\mathrm{d}}
\newcommand{\UH}{\mathit{UH}}
\newcommand{\AT}{\mathit{AT}}

\newcommand{\SAT}{\mathit{SAT}}
\newcommand{\FRL}{\operatorname{FRL}}
\newcommand{\PS}{\operatorname{PS}}
\newcommand{\cF}{\mathcal{F}}
\newcommand{\cJ}{\mathcal{J}}
\newcommand{\sd}{\mathsf{d}}
\newcommand{\can}{\operatorname{can}}

\newcommand{\wan}{\operatorname{wan}}
\newcommand{\CO}{\operatorname{CO}}
\newcommand{\PC}{\operatorname{PC}}
\newcommand{\cE}{\mathcal{E}}
\newcommand{\cS}{\mathcal{S}}
\newcommand{\sF}{\mathsf{F}}
\newcommand{\sJ}{\mathsf{J}}
\newcommand{\diag}{\operatorname{diag}}
\newcommand{\Fekete}{\operatorname{Fekete}}
\newcommand{\Lip}{\operatorname{Lip}}
\newcommand{\diam}{\operatorname{diam}}
\newcommand{\PSU}{\mathit{PSU}}
\newcommand{\PGL}{\mathit{PGL}}
\newcommand{\Capa}{\operatorname{Cap}}
\newcommand{\supp}{\operatorname{supp}}
\newcommand{\Id}{\mathrm{Id}}

\newcommand{\Res}{\operatorname{Res}}

\numberwithin{equation}{section}
\theoremstyle{plain}
\newtheorem{theorem}{Theorem}[section]

\newtheorem{lemma}{Lemma}[section]
\newtheorem{proposition}{Proposition}
\newtheorem{mainth}{Theorem}
\newtheorem{maincoro}{Corollary}
\theoremstyle{definition}
\newtheorem{definition}{Definition}[section]
\newtheorem*{acknowledgement}{Acknowledgement}
\theoremstyle{remark}
\newtheorem{remark}{Remark}[section]
\newtheorem{fact}{Fact}[section]

\begin{document} 

\title[Fekete configuration in non-archimedean dynamics]{
Fekete configuration, quantitative equidistribution
and wandering critical orbits
in non-archimedean dynamics}

\author{Y\^usuke Okuyama}
\address{
Division of Mathematics,
Kyoto Institute of Technology,
Sakyo-ku, Kyoto 606-8585 Japan.}
\curraddr{
UPMC Univ Paris 06, UMR 7586, Institut de
Math\'ematiques de Jussieu, 4 place Jussieu, F-75005 Paris,
France.}

\email{okuyama@kit.ac.jp}

\date{\today}

\dedicatory{Dedicated to Professor Masahiko Taniguchi on his sixtieth birthday}



\maketitle

\begin{abstract}
Let $f$ be a rational function of degree $d>1$
on the projective line over a possibly non-archimedean 
algebraically closed field. A well-known process initiated by Brolin
considers the pullbacks of points under iterates of $f$, and produces
an important equilibrium measure. We define the asymptotic Fekete property
of pullbacks of points, which means that they mirror the equilibrium measure appropriately.
As application, we obtain an error estimate of equidistribution 
of pullbacks of points for $C^1$-test functions 
in terms of the proximity of wandering critical orbits to the initial points, 
and show that the order is $O(\sqrt{kd^{-k}})$ upto a specific exceptional set of capacity $0$ 
of initial points, which is contained in the set of superattracting periodic points
and the omega-limit set of wandering critical points from the Julia set or the presingular 
domains of $f$. As an application in arithmetic dynamics,
together with a dynamical Diophantine approximation, these estimates recover
Favre and Rivera-Letelier's quantitative equidistribution in a purely local manner.
\end{abstract}

\section{Introduction}\label{sec:intro}
Let $K$ or $(K,|\cdot|)$ be an algebraically closed field 
complete with respect to a non-trivial absolute value
(or valuation) $|\cdot|$. The field $K$ is said to be non-archimedean if
it satisfies the strong triangle inequality
\begin{gather*}
 |z-w|\le\max\{|z|,|w|\} 
\end{gather*}
(e.g.\ $p$-adic $\bC_p$), otherwise
$K$ is archimedean and indeed $K\cong\bC$.
When $K$ is non-archimedean,
the projective line $\bP^1=\bP^1(K)=K\cup\{\infty\}$ is 
totally disconnected and non-compact. 
A subset in $K$ is called a {\itshape ball}
if it is written as $\{z\in K;|z-a|\le r\}$ for some center
$a\in K$ and radius $r\ge 0$.
The alternative that two balls in $K$ either nest or are mutually disjoint
induces a partial order over all balls in $K$, which
is nicely visualized by the Berkovich projective line $\sP^1=\sP^1(K)$.
This regards each element of $\sP^1\setminus\{\infty\}$ as an equivalence class of 
nesting balls in $K$, and produces a compact augmentation of $\bP^1$
containing $\bP^1$ as a dense subset. 
A typical point of the hyperbolic space
\begin{gather*}
 \sH^1=\sH^1(K):=\sP^1\setminus\bP^1
\end{gather*}
is a ball of radius $>0$ in $K$, while each point of $K=\bP^1\setminus\{\infty\}$ 
is a ball of radius $0$. On the other hand,
when $K$ is archimedean, $\sP^1$ and $\bP^1$ are identical and $\sH^1=\emptyset$.

Let $f$ be a rational function on $\bP^1$ of degree $d>1$. 
The action of $f$ on $\bP^1$ canonically extends to a continuous,
open, surjective and fiber-discrete endomorphism on $\sP^1$, 
preserving both $\bP^1$ and $\sH^1$. To each $a\in\sP^1$, 
the local degree $\deg_a f$ of $f$ at $a$ also canonically extends.
The exceptional set of the extended $f$ is, as a subset of $\sP^1$, still defined by
\begin{gather*}
 E(f):=\{a\in\bP^1;\#\bigcup_{k\in\bN}f^{-k}(a)<\infty\}.
\end{gather*}
In 1965, Brolin \cite{Brolin} introduced the equilibrium measure $\mu_f$ in the situation that
$f$ is a complex polynomial, which has proved the basis of many later extensions and applications. 
The following equidistribution theorem was established in \cite{Brolin}, \cite{FLM83},
\cite{Lyubich83} for archimedean $K$, and  in \cite{FR09}
generalized to non-archimedean $K$: for each $a\in\sP^1$, 
let $(a)$ be the Dirac measure at $a$ on $\sP^1$. 
If $a\in\sP^1\setminus E(f)$, then the averaged pullbacks 
\begin{gather*}
 \frac{(f^k)^*(a)}{d^k}=\frac{1}{d^k}\sum_{w\in f^{-k}(a)}(\deg_w (f^k))\cdot(w)
\end{gather*}
tends to $\mu_f$ weakly on $\sP^1$.

\begin{definition}
The Berkovich Fatou and Julia sets of $f$ in $\sP^1$
are $\sF(f)$ and $\sJ(f)$, respectively (cf.\ \cite[\S 2.3]{FR09}).
Let $\cF(f)$ and $\cJ(f)$ be 
the classical Fatou and Julia sets of $f$ in $\bP^1$,
which agrees with the intersection of $\sF(f)$ and $\sJ(f)$ with $\bP^1$, respectively.
Let $\SAT(f)$ and $\AT(f)$ be the sets of superattracting and
(possibly super)attracting periodic points of $f$ in $\bP^1$, respectively.
\end{definition} 

In \cite{DOproximity}, for archimedean $K$, the error term of equidistribution 
\begin{gather}
 \left|\int_{\sP^1}\phi\rd\left(\frac{(f^k)^*(a)}{d^k}-\mu_f\right)\right|\label{eq:errorequidist}
\end{gather}
was estimated using a Nevanlinna theoretical covering theory argument:

\begin{theorem}[cf.\ {\cite[Theorem 2 and (4.2)]{DOproximity}}]\label{th:DO}
Let $f$ be a rational function on $\bP^1=\bP^1(\bC)$ of degree $d>1$.
Then for every $C^2$-test function $\phi$ on $\bP^1$,
\begin{gather}
 \left|\int_{\bP^1}\phi\rd\left(\frac{(f^k)^*(a)}{d^k}-\mu_f\right)\right|
=\begin{cases}
  O(d^{-k}) & (a\in\bP^1\setminus\UH(f)),\label{eq:proximity}\\
  O(kd^{-k}) & (a\in\cF(f)\setminus\SAT(f)),\\
  O(\eta^kd^{-k}) & (a\in\bP^1\setminus\SAT(f)),\\
  O((\deg_{a_0} f^k)d^{-k}) & (a\in U_{a_0}, a_0\in\SAT(f))
\end{cases}
\end{gather}
as $k\to\infty$. Here the third estimate applies to any fixed $\eta>1$, 
the fourth one applies to some neighborhood $U_{a_0}$ of any $a_0\in\SAT(f)$,
and the constants implicit in each $O(\cdot)$ are locally uniform on $a$,
and independent of $\eta$ in the third estimate.
The unhyperbolic locus $\UH(f)$ is defined in Definition $\ref{th:semihyp}$ below.
We note that $\UH(f)\cap\cF(f)=\AT(f)$.
\end{theorem}

More interestingly, in arithmetic dynamics,
Favre and Rivera-Letelier \cite[Corollaire 1.6]{FR06}
estimated the order of \eqref{eq:errorequidist}
by $O(\sqrt{kd^{-k}})$ for each $C^1$-test function $\phi$ on $\sP^1$ and
every {\itshape algebraic} $a\in\bP^1\setminus\SAT(f)$
(see (\ref{eq:adelic}) below). 
One of our aims is to establish such a {\itshape quantitative equidistribution}
for general valued fields $K$.
For this purpose, we study quantitatively
the {\itshape $f$-asymptotic Fekete property} 
of $((f^k)^*(a))$ on $\sP^1$
in terms of the {\itshape proximity} of {\itshape wandering}
critical orbits of $f$ to the initial point $a$. 

\begin{definition}
 Let $[\cdot,\cdot]$ be the normalized chordal metric on $\bP^1$.
 A chordal open ball in $\bP^1$ of center $w\in\bP^1$ and radius $r>0$ is 
 $B[w,r]:=\{z\in\bP^1;[z,w]<r\}$. For any subset $S\subset\bP^1$ and $z\in\bP^1$,
 put $[z,S]:=\inf_{w\in S}[z,w]$.
 
 Under the action $f$ on $\bP^1$,
 a point $z_0\in\bP^1$ is said to be {\itshape wandering}
 if $\#\{f^k(z_0);k\in\bN\cup\{0\}\}=\infty$.
 We say $z_0$ to be {\itshape preperiodic} if $z_0$ is not wandering.
 For each $z_0\in\bP^1$, 
 the {\itshape omega limit set} of $(f^k(z_0))$, or of $z_0$, in $\bP^1$ is 
 \begin{gather*}
  \omega(z_0):=\bigcap_{N\in\bN}\overline{\{f^k(z_0);k\ge N\}},
 \end{gather*} 
 where the closure is taken in $\bP^1$ under $[\cdot,\cdot]$
 (then $\lim_{k\to\infty}[f^k(z_0),\omega(z_0)]=0$).
 We say $z_0\in\bP^1$ to be {\itshape pre-recurrent} if
 $\{f^k(z_0);k\in\bN\cup\{0\}\}\cap\omega(z_0)\neq\emptyset$,
 and especially to be {\itshape recurrent} if $z_0\in\omega(z_0)$. Put
\begin{gather*}
 C(f):=\{c\in\bP^1;f'(c)=0\},\\
 C(f)_{\wan}:=\{c\in C(f);c\text{ is wandering under }f\},\\
 \CO(f)_{\wan}:=\{f^k(c);c\in C(f)_{\wan}, k\in\bN\},\\
 \PC(f):=\overline{\{f^k(c);c\in C(f),k\in\bN\}},
\end{gather*}
where in the final definition, the closure is taken in $\bP^1$ under $[\cdot,\cdot]$.
\end{definition}

If $f$ has characteristic $0$, then the Riemann-Hurwitz formula
asserts that there are exactly $(2d-2)$ critical points of $f$ in $\bP^1$
taking into account the multiplicity $(\deg_c f-1)$ of each $c\in C(f)$.

\begin{remark}
 In Section \ref{sec:facts}, we gather a background on the potential
 theory and dynamics on $\sP^1$. For non-archimedean $K$, the chordal metric $[\cdot,\cdot]$ 
 extends to $\sd$ and to $\delta_{\can}$ on $\sP^1$ respectively
 as the small model metric
 and the generalized Hsia kernel with respect to the canonical point $\cS_{\can}\in\sP^1$, 
 and the big model metric $\rho$ is also introduced on $\sH^1$. 
 The equipped (Gel'fand) topology of $\sP^1$ (resp.\ $\sH^1$)
 is strictly weaker than that from $\sd$ (resp.\ $\rho$).
\end{remark}

The $f$-kernel
\begin{gather*}
 \Phi_f(\cS,\cS')=\log\delta_{\can}(\cS,\cS')-g_f(\cS)-g_f(\cS')
\end{gather*}
on $\sP^1$ is introduced in (\ref{eq:Arakerov}) in Section \ref{sec:facts},
where $g_f$ is the dynamical Green function of $f$ on $\sP^1$. We note that
$-\Phi_f$ agrees with the Arakelov Green (kernel) function of $f$ in \cite[\S10.2]{BR10},
and that
\begin{gather*}
 \{(\cS,\cS')\in\sP^1\times\sP^1;\Phi_f(\cS,\cS')=-\infty\}=
\diag_{\bP^1}:=\{(z,z)\in\bP^1\times\bP^1; z\in\bP^1\}.
\end{gather*}
A dynamical Favre and Rivera-Letelier bilinear form is
\begin{gather}
 (\mu,\mu')_f:=-\int_{\sP^1\times\sP^1\setminus\diag_{\bP^1}}\Phi_f(z,w)\rd(\mu\times\mu')(z,w)\label{eq:FRL}
 \end{gather}
for Radon measures $\mu,\mu'$ on $\sP^1$ (if exists).
For each $a\in\bP^1$ and each $k\in\bN$, the following quantities
\begin{gather}
\cE_f(k,a):=-\frac{1}{d^{2k}}((f^k)^*(a),(f^k)^*(a))_f,\label{eq:errorFekete}
\\ 
\notag\eta_{a,k}=\eta_{a,k}(f):=\max_{w\in f^{-k}(a)}\deg_w(f^k),\\
\notag\begin{aligned}
 D_{a,k}=&D_{a,k}(f):=((f^k)^*(a)\times(f^k)^*(a))(\diag_{\bP^1})\\
 =&\int_{\sP^1}\deg_w(f^k)\rd((f^k)^*(a))(w)
 =\sum_{w\in f^{-k}(a)}(\deg_w(f^k))^2\in[d^k,d^k\eta_{a,k}]
\end{aligned}
\end{gather}
are fundamental. 

\begin{fact}\label{th:cocycle}
 If $K$ has characteristic $0$, then 
 $E(f)$ and $\SAT(f)$ are respectively characterized as follows ({cf. \cite[Lemma 1]{ES90}}):
 \begin{gather}
\label{eq:non-exceptional}
  \limsup_{j\to\infty}\eta_{a,j}^{1/j}
 \begin{cases}
 \le(d^3-1)^{1/3} & (a\in\bP^1\setminus E(f))\\
 =d & (a\in E(f))
 \end{cases},\\
\notag \sup_{j\in\bN}\eta_{a,j}
 \begin{cases}
 \le d^{2d-2} & (a\in\bP^1\setminus\SAT(f))\\
 =\infty & (a\in\SAT(f))
 \end{cases}.
\end{gather}
\end{fact}

One of our principal results is the following 
estimates of $\cE_f(k,a)$ in terms of the proximity of wandering critical orbits $\CO(f)_{\wan}$ 
to the initial $a\in\sP^1$: 

\begin{mainth}\label{th:algebraic}
Let $f$ be a rational function on $\bP^1=\bP^1(K)$ of degree $d>1$. Then
for every $a\in\sH^1$ and every $k\in\bN$, 
 \begin{gather}
  |\cE_f(k,a)|\le\frac{\Phi_f(a,a)}{d^k},\label{eq:non-classical}
 \end{gather}
and $a\mapsto\Phi_f(a,a)$ is locally bounded on
$\sH^1$ under $\rho$. If in addition $K$ has characteristic $0$, 
 then there is $C_f>0$ such that for every $a\in\bP^1$ and every $k\in\bN$, 
\begin{multline}
  -\frac{1}{d^k}\sum_{j=1}^k\left(\eta_{a,j}\sum_{c\in C(f)\setminus
 f^{-j}(a)}\frac{1}{d^j}\log\frac{1}{[f^j(c),a]}\right)
 -\frac{C_f}{d^k}\sum_{j=1}^k\eta_{a,j}
  -\frac{C_{f,a}}{d^k}\\
   \le\cE_f(k,a)\le-\frac{1}{d^k}\sum_{j=1}^k\left(\sum_{c\in C(f)\setminus
  f^{-j}(a)}\frac{1}{d^j}\log\frac{1}{[f^j(c),a]}\right)
  +\frac{C_f}{d^k}\sum_{j=1}^k\eta_{a,j}+\frac{C_{f,a}}{d^k}.\label{eq:algebraic}
\end{multline}
Here the extra constant $C_{f,a}>0$ is independent of $k$, and
even vanishes if $a\in\bP^1\setminus\CO(f)_{\wan}$.
The sums over $C(f)\setminus f^{-j}(a)$
take into account the multiplicity $(\deg_c f-1)$ of each $c$.
\end{mainth}

The constants $C_f$ and $C_{f,a}$ are concretely 
given in Section \ref{eq:proof} below.

\begin{remark}
 The upper estimate of $\cE_f(k,a)$ in (\ref{eq:algebraic})
 improves \cite[Theorem 1.1]{Baker06} and \cite[Propositions 2.8, 4.9]{FR06}
 by the first {\itshape proximity} term of wandering critical orbits 
 to the initial $a\in\bP^1$.
\end{remark}

In Section \ref{sec:facts}, we introduce the notion of
$f$-asymptotic Fekete configuration on $\sP^1$
to sequences of positive measures whose supports consist of finitely many points
in $\sP^1$. Here we mention that
\begin{align*}
E_{\Fekete}(f)
:=&\{a\in\sP^1;((f^k)^*(a))\text{ is not }f\text{-asymptotically Fekete on }\sP^1\}\\
=&E(f)\cup\{a\in\sP^1\setminus E(f);\lim_{k\to\infty}\cE_f(k,a)\neq 0\}
\end{align*}
(see Remark \ref{th:exceptionalfekete} below).
We call $E_{\Fekete}(f)$ the {\itshape non-Fekete locus} of $f$.

\begin{fact}\label{th:classification}
Suppose that $K$ has characteristic $0$.
By the classification of cyclic Fatou components 
(essentially due to Fatou, cf. \cite[Theorem 5.2]{Milnor3rd}) and its non-archimedean counterpart due to Rivera-Letelier (\cite[Th\'eor\`eme de Classification]{Juan03}),
each Berkovich Fatou component $U$ of $f$ is either 
\begin{itemize}
 \item a wandering domain, that is,
       for any distinct $m,n\in\bN\cup\{0\}$, $f^m(U)\cap f^n(U)=\emptyset$, or
 \item a component of (super)attracting or parabolic basin, or 
 \item the other, which we call a {\itshape presingular domain} of $f$. 
\end{itemize}
By definition, for every presingular domain $U$, there are $m\in\bN\cup\{0\}$ and
$p\in\bN$ such that $f^p(f^m(U))=f^m(U)$, and by \cite[Proposition 2.16]{FR09},
the restriction of $f^p$ to $f^m(U)$ is injective.
Historically, a cyclic (Berkovich) Fatou component of period $p$ on which $f^p$ 
is injective was called a singular domain (un domaine singulier) of $f$ 
(cf.\ \cite[\S 28]{Fatou1920deux}). For archimedean $K$,
each singular domain is either a Siegel disk or an Herman ring.
\end{fact}

\begin{definition}
Let $\PS(f)$ be the the union of all presingular domains of $f$. 
\end{definition}
For $K$ having characteristic $0$,
put $C_0:=-\min_{c\in C(f)\cap\cF(f)}\log[c,\cJ(f)]+2\sup_{\bP^1}|g_f|>0$ and
\begin{gather}
E_{\wan}(f)
=:\bigcup_{c\in C(f)_{\wan}\cap(\cJ(f)\cup\PS(f))}
\bigcap_{N\in\bN}\bigcup_{j\ge N}B[f^j(c),\exp(-C_0d^j)],\label{eq:polarset}
\end{gather}
which is 
of finite Hyllengren measure for increasing $(d^j)\subset\bN$, so 
of capacity $0$ (Lemma \ref{th:polar}). 

The estimates of $\cE_f(k,a)$ from below and above in (\ref{eq:algebraic})
are respectively essential in estimating the size of $E_{\Fekete}(f)$
from above and below. 

\begin{mainth}\label{th:characterization}
Let $f$ be a rational function on $\bP^1=\bP^1(K)$ of degree $d>1$,
and suppose that $K$ has characteristic $0$. Then
\begin{gather}
 E_{\Fekete}(f)\setminus E(f)\subset E_{\wan}(f)\setminus E(f).\label{eq:qualitative}
\end{gather} 
Moreover, if $c\in C(f)_{\wan}$ is pre-recurrent $($then this $c$ belongs
to $\cJ(f)\cup\PS(f)$ from Fact $\ref{th:classification})$,
$E_{\Fekete}(f)\cap\omega(c)$ is $G_{\delta}$-dense in $\omega(c)$
under $[\cdot,\cdot]$.
\end{mainth}

Let us come back to estimating (\ref{eq:errorequidist}). 
The following is a version of Favre and Rivera-Letelier \cite[Th\'eor\`eme 7]{FR06}.

\begin{proposition}\label{th:cauchy}
Let $f$ be a rational function on $\bP^1=\bP^1(K)$ of degree $d>1$. 
 Then for every $a\in\sH^1$, every $C^1$-test function $\phi$ on $\sP^1$ and every $k\in\bN$,
\begin{gather}
 \left|\int_{\sP^1}\phi\rd\left(\frac{(f^k)^*(a)}{d^k}-\mu_f\right)\right|
\le\langle\phi,\phi\rangle^{1/2}\sqrt{|\cE_f(k,a)|}.\label{eq:cauchyhyp}
\end{gather}
 Moreover, there is $C_{\FRL}>0$ such that
 for every $a\in\bP^1$, every $C^1$-test function $\phi$ on $\sP^1$ and every $k\in\bN$, 
\begin{multline}
 \left|\int_{\sP^1}\phi\rd\left(\frac{(f^k)^*(a)}{d^k}-\mu_f\right)\right|\\
\le C_{\FRL}\max\{\Lip(\phi),
 \langle\phi,\phi\rangle^{1/2}\}\sqrt{|\cE_f(k,a)|+kd^{-2k}D_{a,k}}.\label{eq:cauchy}
\end{multline}
Here $\Lip(\phi)$ is the Lipschitz constant 
of the restriction of $\phi$ to $\bP^1$ under $[\cdot,\cdot]$,
and $\langle\phi,\phi\rangle^{1/2}$ is the
Dirichlet norm of $\phi$.
\end{proposition}

\begin{remark}\label{th:lipschitz}
 For the $C^1$-regularity of test functions on $\sP^1$ in non-archimedean $K$ case,
 see Section \ref{sec:facts}. The dependence of $C_{\FRL}$ on $f$
 will be seen concretely in Section \ref{sec:error}. 
 If $K$ is archimedean ($\cong\bC$), then
 each $C^1$-test function $\phi$ satisfies
 $\langle\phi,\phi\rangle^{1/2}\le\Lip(\phi)$, 
 and any Lipschitz continuous test function on $\bP^1$
 under $[\cdot,\cdot]$ is approximated by $C^1$-test functions on $\bP^1$
 in the Lipschitz norm. Hence the estimate (\ref{eq:cauchy}) 
 extends to every Lipschitz continuous test function $\phi$ on $\bP^1$ under $[\cdot,\cdot]$
 (and every $a\in\bP^1$ and every $k\in\bN$) as
\begin{gather*}
 \left|\int_{\sP^1}\phi\rd\left(\frac{(f^k)^*(a)}{d^k}-\mu_f\right)\right|
 \le C_{\FRL}\Lip(\phi)\sqrt{|\cE_f(k,a)|+kd^{-2k}D_{a,k}}.
\end{gather*}
 As a consequence, (\ref{eq:general}) and (\ref{eq:derived}) below also extend similarly.
\end{remark}

The first principal estimate of (\ref{eq:errorequidist}) is

\begin{mainth}\label{th:order}
 Let $f$ be a rational function on $\bP^1=\bP^1(K)$ of degree $d>1$. 
 Then for every $a\in\sH^1$, every $C^1$-test function $\phi$ on $\sP^1$ 
 and every $k\in\bN$,
 \begin{gather*}
 \left|\int_{\sP^1}\phi\rd\left(\frac{(f^k)^*(a)}{d^k}-\mu_f\right)\right|
 \le \langle\phi,\phi\rangle^{1/2}\sqrt{\Phi_f(a,a)d^{-k}},
 \end{gather*}
 and $a\mapsto\Phi_f(a,a)$ is locally bounded on $\sH^1$ under $\rho$.
 If in addition $K$ has characteristic $0$, then 
 for every $a\in\bP^1$,
 every $C^1$-test function $\phi$ on $\sP^1$ and every $k\in\bN$,
\begin{multline}
 \left|\int_{\sP^1}\phi\rd\left(\frac{(f^k)^*(a)}{d^k}-\mu_f\right)\right|
 \le C_{\FRL}\max\{\Lip(\phi),\langle\phi,\phi\rangle^{1/2}\}\times\\
 \sqrt{(2d-2)\left(\max_{j\in\{1,\ldots,k\}}\max_{c\in C(f)\setminus f^{-j}(a)}\frac{1}{d^j}\log\frac{1}{[f^j(c),a]}\right)+C_f+C_{f,a}+1}\\
\times\sqrt{\max\left\{\frac{1}{d^k}\sum_{j=1}^k\eta_{a,j},k\frac{D_{a,k}}{d^{2k}}\right\}}.\label{eq:general}
\end{multline}
Here the constants $C_f$ and $C_{f,a}$ appear in Theorem $\ref{th:algebraic}$,
and the sum over $C(f)\setminus f^{-j}(a)$
takes into account the multiplicity $(\deg_c f-1)$ of each $c$.
\end{mainth}

Recall the definition (\ref{eq:polarset}) of $E_{\wan}(f)$,
which is of capacity $0$. The second principal estimate of (\ref{eq:errorequidist}) is

\begin{mainth}\label{th:derived}
 Let $f$ be a rational function on $\bP^1=\bP^1(K)$ of degree
 $d>1$, and suppose that $K$ has characteristic $0$. Then 
 for every $a\in\bP^1\setminus E_{\wan}(f)$, there is $C>0$ such that 
 every $C^1$-test function $\phi$ on $\sP^1$ and every $k\in\bN$,
\begin{multline}
 \left|\int_{\sP^1}\phi\rd\left(\frac{(f^k)^*(a)}{d^k}-\mu_f\right)\right|\\
 \le C\max\{\Lip(\phi),\langle\phi,\phi\rangle^{1/2}\}
\sqrt{\max\left\{\frac{1}{d^k}\sum_{j=1}^k\eta_{a,j},k\frac{D_{a,k}}{d^{2k}}\right\}}
\label{eq:derived}.
\end{multline}
 Furthermore, for every $a_0\in\bP^1\setminus\PC(f)$ and every $r_0\in(0,[a_0,\PC(f)])$, 
there is $C'>0$ such that
 for every $C^1$-test function $\phi$ on $\sP^1$ and every $k\in\bN$,
\begin{gather*}
 \sup_{a\in B[a_0,r_0]}\left|\int_{\sP^1}\phi\rd\left(\frac{(f^k)^*(a)}{d^k}-\mu_f\right)\right|
 \le C'\max\{\Lip(\phi),\langle\phi,\phi\rangle^{1/2}\}\sqrt{kd^{-k}}.
\end{gather*}
\end{mainth}

\begin{remark}\label{th:better}
 Suppose that $K$ has characteristic $0$. For every $a\in\bP^1$ and every $k\in\bN$,
\begin{gather}
 \max\left\{\frac{1}{d^k}\sum_{j=1}^k\eta_{a,j},k\frac{D_{a,k}}{d^{2k}}\right\}
\le kd^{-k}\eta_{a,k},\label{eq:generic}
\end{gather}
 which is asymptotically optimal as $k\to\infty$ 
 if $a\in\bP^1\setminus\SAT(f)$. By (\ref{eq:non-exceptional}), we also 
 have $\sup_{k\in\bN}\eta_{a,k}\le d^{2d-2}$ if $a\in\bP^1\setminus\SAT(f)$.
 If $a\in\SAT(f)\setminus E(f)$, we have a better order estimate
 $O(d^{-k}\eta_{a,k})$ of the left hand side as $k\to\infty$.
 Note that $(\SAT(f)\setminus E(f))\cap E_{\wan}(f)=\emptyset$, and
 from $\#\SAT(f)<\infty$, that $E_{\wan}(f)\cup\SAT(f)$ is still of capacity $0$.
\end{remark}

Let us recover the {\itshape arithmetic} quantitative equidistribution theorem
in a {\itshape purely local} manner:
here, let $k$ be a number field or a function field 
with a place $v$, and $\overline{k}$ the algebraic closure of $k$.
Under the {\itshape arithmetic setting}, that is,
\begin{itemize}
 \item setting $K=\bC_v$ and 
 \item assuming that $f$ has its coefficients in $k$,
\end{itemize}
the dynamical Diophantine approximation due to Silverman \cite[Theorem E]{Silverman93} and
Szpiro and Tucker \cite[Proposition 4.3]{ST05} asserts that
for every $a\in\bP^1(\overline{k})\setminus E(f)$ and every wandering $z\in\bP^1(\overline{k})$,
\begin{gather}
 \lim_{n\to\infty}\frac{1}{d^n}\log[f^n(z),a]_v=0.\label{eq:diophantine}
\end{gather}
(the dependence of $[\cdot,\cdot], E_{\wan}(f)$ and $\mu_f$ on $v$ is emphasized by the suffix $v$).
Since $(E(f)\subset\SAT(f)\subset)C(f)\subset\bP^1(\overline{k})$,
a consequence of (\ref{eq:diophantine}) is
\begin{gather}
 E_{\wan}(f)_v\cap\bP^1(\overline{k})\subset E(f),\label{eq:empty}
\end{gather}
and Theorem \ref{th:derived} recovers
Favre and Rivera-Letelier's arithmetic quantitative equidistribution theorem 
\cite[Corollaire 1.6]{FR06}: under the above 
arithmetic setting,
{\itshape 
for every $a\in\bP^1(\overline{k})\setminus E(f)$, there is $C>0$ such that 
every $C^1$-test function $\phi$ on $\sP^1(\bC_v)$ and every $n\in\bN$,
\begin{gather}
 \left|\int_{\sP^1(\bC_v)}\phi\rd\left(\frac{(f^n)^*(a)}{d^n}-\mu_{f,v}\right)\right|
 \le C\max\{\Lip(\phi),\langle\phi,\phi\rangle^{1/2}\}
\sqrt{kd^{-k}\eta_{a,k}}\label{eq:adelic}
\end{gather}}
($\Lip(\phi)$ and $\langle\phi,\phi\rangle^{1/2}$ also depend on $v$).

\begin{remark}
 Theorem \ref{th:characterization} with (\ref{eq:empty}) also implies
\begin{gather*}
 E_{\Fekete}(f)_v\cap\bP^1(\overline{k})=E(f)
\end{gather*}
($E_{\Fekete}(f)$ also depends on $v$).
This is a substance of the {\itshape adelic equidistribution theorem} 
due to Baker and Rumely \cite[Theorems 2.3 and 4.9]{BR06}, 
Chambert-Loir \cite[Th\'eor\`eme 3.1]{ChambertLoir06} and
Favre and Rivera-Letelier \cite[Th\'eor\`emes 3 et 7]{FR06}.
\end{remark}

It seems interesting to determine when $E_{\Fekete}(f)=E(f)$ holds. 
By Theorem \ref{th:characterization},
this is the case under the condition
\begin{gather}
 C(f)_{\wan}\cap(\cJ(f)\cap\PS(f))=\emptyset.\label{eq:hyperbolic}
\end{gather}
Let us study this problem further
under the assumption that $K$ is archimedean. 
For complex dynamics, see \cite{Milnor3rd}.
 
\begin{definition}\label{th:semihyp}
 We say $f$ to be {\itshape semihyperbolic at} $a\in\bP^1$ if
 there is $r>0$ such that
 $\sup_{k\in\bN}\max_{V^{-k}}\deg(f^k:V^{-k}\to B[a,r])<\infty$,
 where $V^{-k}$ ranges over all components of $f^{-k}(B[a,r])$.
 The unhyperbolic locus $\UH(f)$ is 
 the set of all points at which $f$ is not semihyperbolic
 (as in Theorem \ref{th:DO}). 
\end{definition}

From Ma\~n\'e \cite[Theorem II, Corollary]{ManeFatou}, 
$\UH(f)\cap\cJ(f)$ agrees with
\begin{gather*}
\left(\bigcup_{c\in C(f)\cap\cJ(f),\text{ recurrent}}\omega(c)\right)
\cup\{\text{parabolic periodic points of }f\},
\end{gather*}
and each of Cremer periodic points of $f$ 
and components of the boundaries of Siegel disks and Herman rings of $f$
is contained in $\omega(c)$ for some recurrent $c\in C(f)_{\wan}\cap\cJ(f)$
(a generalization of a theorem of Fatou).
Hence by the final assertion of Theorem \ref{th:characterization},

\begin{maincoro}\label{th:rotation}
 Let $f$ be a rational function on $\bP^1(\bC)$ of degree $>1$.
 If either there is a Cremer periodic point or $\PS(f)\neq\emptyset$,
 then $E(f)\subsetneq E_{\Fekete}(f)$. Indeed,
 $E_{\Fekete}(f)\cap\cJ(f)$ is uncountable.
\end{maincoro}

\begin{definition}
 We say $f$ to be {\itshape geometrically finite} if 
 $C(f)_{\wan}\cap\cJ(f)=\emptyset$, or
 to be {\itshape semihyperbolic} if $\UH(f)\cap\cJ(f)=\emptyset$.
\end{definition}

From Ma\~n\'e's theorem,
if $f$ is either geometrically finite or semihyperbolic,
then $f$ has no Cremer periodic points and $\PS(f)=\emptyset$.
Hence the condition \eqref{eq:hyperbolic} is formally improved as

\begin{maincoro}\label{th:archimedean}
 If $f$ is geometrically finite, then $E_{\Fekete}(f)=E(f)$. 
\end{maincoro}

\begin{remark}\label{th:intermidiate}
 It is possible to construct a semihyperbolic real bimodal cubic polynomial $f$ 
 with one (non-recurrent and) pre-recurrent critical point in $J(f)$,
 so that $E_{\Fekete}(f)\cap\cJ(f)\neq\emptyset$,
 using the kneading theory for bimodal maps having
 one strictly preperiodic critical point 
 developed by Mihalache \cite{Mihalache_construction}. Alternatively, 
 the $1$-parameter family of polynomials in \cite[\S 6.1]{PRS03}
 also produces a semihyperbolic (complex) polynomial with the same property.
\end{remark}

In Section \ref{sec:facts}, we gather background material on dynamics and
potential theory on $\sP^1$. In Section \ref{eq:proof} we show
Theorems \ref{th:algebraic}, \ref{th:characterization}, \ref{th:order} and \ref{th:derived}.
In Section \ref{sec:error} we give a proof of Proposition
\ref{th:cauchy}.

\section{Background}\label{sec:facts}
For the foundation including the construction of Laplacian $\Delta$ on $\sP^1$, 
see \cite{BR10}, \cite{FR06}, \cite{Jonsson12}. 

We denote the origin of $K^2$ by $0$. Both the maximum
norm on $K^2$ for non-archimedean $K$ and
the Euclidean norm on $K^2$ for archimedean $K$ are denoted by the same $|\cdot|$.
Let $\pi:K^2\setminus\{0\}\to\bP^1=\bP^1(K)$ be the canonical projection. 
Put $(z_0,z_1)\wedge(w_0,w_1):=z_0w_1-z_1w_0$ on $K^2\times K^2$. 
The normalized chordal distance $[z,w]$ on $\bP^1$ is 
\begin{gather*}
 [z,w]:= |p\wedge q|/(|p|\cdot|q|)
\end{gather*} 
if $p\in\pi^{-1}(z)$ and $q\in\pi^{-1}(w)$.
For non-archimedean $K$,
the canonical (or Gauss) point $\cS_{\can}$ of $\sP^1$ is the unit ball $\{z\in
K;|z|\le 1\}$ in $K$. 

Suppose that $K$ is non-archimedean. 
The strong triangle inequality implies that
any point of a ball $\cS$ in $K$ can be
its center, and its radius and the diameter $\diam\cS$ are identical.
For balls $\cS,\cS'$ in $K$,
$\cS\wedge\cS'$ is the smallest ball in
$K$ containing $\cS\cup\cS'$. 
The {\itshape Hsia kernel}
$\delta_{\infty}(\cdot,\cdot)$ on $\sP^1\setminus\{\infty\}$ is 
an extension of $|z-w|$ for $z,w\in K$ so that
\begin{gather}
 \delta_{\infty}(\cS,\cS') =\diam(\cS\wedge\cS')\label{eq:Hsia}
\end{gather}
for balls $\cS,\cS'$ in $K$. 

The {\itshape big model metric $\rho$ on} $\sH^1$ is an extension of
the {\itshape modulus} 
\begin{gather*}
 \rho(\cS,\mathcal{S'}):=\log(\diam\cS'/\diam\cS)
\end{gather*}
for balls $\cS\subset\cS'$ in $K$ as a path-length metric on $\sH^1$.
Let $\rd\rho$ be the $1$-dimensional Hausdorff measure on $\sH^1$ under
$\rho$. A function $\phi$ on $\sP^1$ is said to be $C^1$ 
if it is locally constant on $\sP^1$ except for a finite sub-tree
in $\sH^1$ and if $\phi'=\rd\phi/\rd\rho$ exists and is
continuous there. The Dirichlet norm of $\phi$
is defined by $\langle\phi,\phi\rangle^{1/2}$,
where $\langle\phi,\phi\rangle$ is the integration of $(\phi')^2$ in $\rd\rho$ over $\sH^1$
(\cite[\S 5.5]{FR06}). 

\begin{fact}
 Since $\PGL(2,K)$ is the linear fractional isometry group on $\sH^1$ under $\rho$
 (\cite[\S 3.4]{FR06}),
 the Dirichlet norm $\langle\phi,\phi\rangle^{1/2}$ of $\phi$ is $\PGL(2,K)$-invariant in that
 for every $h\in\PGL(2,K)$, $\langle h^*\phi,h^*\phi\rangle=\langle\phi,\phi\rangle$.
\end{fact}

Put $|\cS|:=\delta_{\infty}(\cS,0)=\sup_{z\in\cS}|z|$ for each ball $\cS$ in $K$.
The {\itshape small model metric} $\sd$ on $\sP^1$ 
is an extension of $[\cdot,\cdot]$ as a path-length metric on $\sP^1$ so that
\begin{gather*}
 \sd(\cS,\mathcal{S'})=\frac{\diam(\cS\wedge\cS')}{\max\{1,|\cS|\}\max\{1,|\cS'|\}}
 -\frac{1}{2}\left(\frac{\diam\cS}{\max\{1,|\cS|\}^2}+\frac{\diam\cS'}{\max\{1,|\cS'|\}^2}\right)
\end{gather*}
for balls $\cS,\cS'$ in $K$. 
For the potential theory on $\sP^1$, another extension of $[\cdot,\cdot]$ is essential.
The {\itshape generalized Hsia kernel} $\delta_{\can}(\cdot,\cdot)$ on $\sP^1$
with respect to $\mathcal{S}_{\can}$ is an extension of $[\cdot,\cdot]$ so that
\begin{gather*}
 \delta_{\can}(\cS,\cS')=\frac{\diam(\cS\wedge\cS')}{\max\{1,|\cS|\}\max\{1,|\cS'|\}}
\end{gather*}
for balls $\cS,\cS'$ in $K$. Then
$\{a\in\sP^1;\delta_{\can}(a,a)=0\}=\bP^1$.

For archimedean $K\cong\bC$,
we put $\delta_{\infty}(z,w):=|z-w|$ ($z,w\in K$), and
define $\delta_{\can}(\cdot,\cdot)$ by $[\cdot,\cdot]$
as convention. 
We define $C^1$-regularity and the Lipschitz continuity under $[\cdot,\cdot]$
of functions on $\bP^1$ as usual.

\begin{fact}\label{th:invariance}
 For archimedean $K$, 
 $\PSU(2,K)$ is the linear fractional isometry group on $\bP^1$ under
 $[\cdot,\cdot]$. For non-archimedean $K$, so is $\PGL(2,\mathcal{O}_K)$,
 where $\mathcal{O}_K=\mathcal{O}_{K,v}$ is the ring of integers of $K$ (cf.\ \cite[\S 1]{Benedetto03}). 
 The {\itshape chordal kernel} $\log\delta_{\can}(\cdot,\cdot)$ is interpreted as
 a Gromov product on $\sP^1$ by the equality
 \begin{gather}
  \log\delta_{\can}(\cS,\cS')=-\rho(\cS'',\cS_{\can})\label{eq:Gromov}
 \end{gather}
 for balls $\cS,\cS'\in\sP^1$,
 where $\cS''$ is the unique point in $\sP^1$ such that
 it lies between $\cS$ and $\cS'$, between $\cS'$ and $\cS_{\can}$,
 and between $\cS$ and $\cS_{\can}$.
\end{fact}

Let $f$ be a rational function of degree $d>1$ on $\bP^1$.
A {\itshape lift} $F$ of $f$ is a homogeneous polynomial map $F:K^2\to K^2$ 
such that $\pi\circ F=f\circ\pi$ and that $F^{-1}(0)=\{0\}$, i.e., non-degenerate, 
and is unique upto multiplication in $K^*=K\setminus\{0\}$.
The action of $f$ on $\bP^1$ canonically extends to a continuous, open, 
surjective and fiber-discrete endomorphism on $\sP^1$, and to each $a\in\sP^1$, 
the local degree $\deg_a f$ of $f$ at $a$ also canonically extends.
The pullback $f^*$ and push-forward $f_*$ 
on the space of continuous functions on $\sP^1$ and
that of Radon measures on $\sP^1$ are defined as usual. 
The dynamical Green function of $F$ is the uniform limit 
\begin{gather*}
 g_F:=\sum_{j=0}^{\infty}\frac{(f^j)^*T_F}{d^j}
\end{gather*}
on $\sP^1$, where $T_F:=(\log|F|)/d-\log|\cdot|$
descends to $\bP^1$ and extends continuously to $\sP^1$.
For other lifts of $f$, which are written as $cF$ for some $c\in K^*$, 
the homogeneity of $F$ implies
\begin{gather}
 g_{cF}=g_F+\frac{1}{d-1}\log|c|.\label{eq:ambiguitygreen}
\end{gather}

\begin{fact}
 The uniform limit
 \begin{gather}
 G^F:=(g_F|\bP^1)\circ \pi+\log|\cdot|=\lim_{k\to\infty}\frac{1}{d^k}\log|F^k|\label{eq:escaping}
 \end{gather}
 on $K^2\setminus\{0\}$ is called the escaping rate function of $F$, and
 satisfies that $G^F\circ F=d\cdot G^F$.
\end{fact}

The Laplacian $\Delta$ on $\sP^1$
is normalized so that for every $w\in\sP^1\setminus\{\infty\}$,
\begin{gather*}
 \Delta\log\delta_{\infty}(\cdot,w)=(w)-(\infty).
\end{gather*}
For non-archimedean $K$,
see \cite[\S 5]{BR10}, \cite[\S7.7]{FJbook}, \cite[\S 3]{Thuillierthesis}: 
in \cite{BR10} the opposite sign convention on $\Delta$ is adopted.

The equilibrium measure of $f$ is 
\begin{gather*}
 \mu_f:=
\begin{cases}
 \Delta g_F+(\cS_{\can}) & (K \text{ is non-archimedean}),\\
 \Delta g_F+\omega & (K \text{ is archimedean}),
\end{cases}
\end{gather*} 
where for archimedean $K$, $\omega$ is the normalized Fubini-Study area element on $\bP^1$.
Then $\mu_f$ is a probability Radon measure on $\sP^1$ and independent of choices of $F$.
Moreover, $\mu_f$ has no atom in $\bP^1$ and is balanced and invariant 
under $f$ in that
\begin{gather*}
 \frac{f^*\mu_f}{d}=\mu_f=f_*\mu_f. 
\end{gather*}

\begin{fact}\label{th:Holder}
 The extended endomorphism $f$ on $(\sP^1,\sd)$ is $M$-Lipschitz continuous
 for some $M>d$, and the function $T_F$ is Lipschitz continuous on $(\sP^1,\sd)$
 (for non-archimedean $K$, see \cite[Theorems 10, 13]{KS07}). 
 Put $\kappa:=(\log d)/(\log M)\in(0,1)$. Then
 there is $C>0$ such that for every $\alpha\in(0,\kappa)$,
 every $z\in\sP^1$ and every $w\in\sP^1$,
\begin{gather*}
 |g_F(z)-g_F(w)|
 \le C\left(\sum_{j=0}^{\infty}(\frac{M^{\alpha}}{d})^j\right)\sd(z,w)^{\alpha},
\end{gather*}
 that is, $g_F$ is $\alpha$-H\"older continuous on $\sP^1$ under $\sd$.
 With more effort, it is possible to shows that $g_F$ is 
 indeed $\kappa$-H\"older continuous on $\sP^1$ under $\sd$ (cf.\ \cite[\S6.6]{FR06}). 
 The uniform limit
 $g_{\infty}:=\sum_{j=0}^{\infty}(f^j)^*T_{F,\infty}/d^j$ on $\sP^1$,
 where $|\cdot|_{\max}$ is the maximum norm on $K^2$ and
 $T_{F,\infty}:=(\log|F|_{\max})/d-\log|\cdot|_{\max}$ descends to $\bP^1$
 and extends continuously to $\sP^1$,
 agrees with $g_F$ for non-archimedean $K$. For archimedean $K$,
 $g_{\infty}$ is also $\kappa$-H\"older continuous on $\bP^1$ under $[\cdot,\cdot]$
 and satisfies that $\Delta g_{\infty}=\mu_f-\lambda_1$, where
 $\lambda_1$ is the normalized Lebesgue measure on the unit circle in $K\cong\bC$.
\end{fact} 

We refer \cite{Brelot67} for an abstract potential theory on
a locally compact topological space equipped with 
an upper semicontinuous kernel. For non-archimedean $K$, see \cite{BR10},
and for archimedean $K$, see \cite{ST97}, \cite{Tsuji59}.

\begin{fact}\label{th:chordal}
 The {\itshape chordal capacity} of a Borel set $E$ in $\sP^1$ is defined by
 \begin{gather*}
 \Capa_{\can}(E):=\exp\left(\sup_{\mu}\int_{\sP^1\times\sP^1}\log\delta_{\can}(z,w)\rd(\mu\times\mu)(z,w)\right),
 \end{gather*}
 where the supremum is taken over all probability Radon measures $\mu$ on $\sP^1$
 with $\supp\mu\subset E$. If $E$ is compact, then the $\sup$ can be
 replaced by $\max$, and if in addition $\Capa_{\can}(E)>0$,
 then the $\max$ is taken at the unique $\mu$, and the {\itshape chordal potential}
 \begin{gather*}
  U_{\can,\mu}(z):=\int_{\sP^1}\log\delta_{\can}(z,w)\rd\mu(w)
 \end{gather*}
 for this $\mu$ is bounded from below by $\log\Capa_{\can}(E)$ on $\sP^1$. 
\end{fact}

\begin{definition}
 A Borel set $E$ in $\sP^1$ is {\itshape of capacity} $0$
 if $\Capa_{\can}(E)=0$.
\end{definition}

For example, a subset in $\sP^1$ of finite Hyllengren measure 
for an increasing sequence in $\bN$ is
of logarithmic measure $0$ under $\sd$, 
so of capacity $0$ (cf.\ \cite[\S 2]{Sodin92}).
In particular,

\begin{lemma}\label{th:polar}
 Suppose that $K$ has characteristic $0$. Then
 $E_{\wan}(f)$ in \eqref{eq:polarset} is of capacity $0$.
\end{lemma}

\begin{proof}
 We include a direct proof. Suppose $\Capa_{\can}(E_{\wan}(f))>0$. Then 
 there is a probability Radon measure $\mu$ with $\supp\mu\subset E_{\wan}(f)$
 such that the chordal potential $U_{\can,\mu}$ is bounded (from below) on $\sP^1$
 (Fact \ref{th:chordal}).
 Put $C:=-\inf_{z\in\sP^1}U_{\can,\mu}(z)(\ge 0)$. 
 For every $z\in E_{\wan}(f)$ and every $r\in(0,1)$,
\begin{gather*}
 (\log r)\mu(B[z,r])\ge\int_{B[z,r]}\log\delta_{\can}(z,w)\rd\mu(w)\ge -C, 
\end{gather*}
that is, $\mu(B[z,r])\le C(\log(r^{-1}))^{-1}$. Hence for every $N\in\bN$, 
\begin{multline*}
 \mu(E_{\wan}(f))
\le \sum_{c\in C(f)\cap(\cJ(f)\cup\PS(f))}\sum_{j\ge N}\mu(B[f^j(c),C_0\exp(-d^j)])\\
\le (2d-2)\sum_{j\ge N}C(d^j-\log C_0)^{-1},
\end{multline*}
so $\mu(E_{\wan}(f))=0$ (as $N\to\infty$). This contradicts that $\mu(E_{\wan}(f))=1$.
\end{proof}

Let us introduce the dynamically weighted $F$-kernel 
\begin{gather*}
 \Phi_F(z,w):=\log\delta_{\can}(z,w)-g_F(z)-g_F(w)
\end{gather*} 
on $\sP^1$, which satisfies that for every $w\in\sP^1$,
\begin{gather}
 \Delta\Phi_F(\cdot,w)=(w)-\mu_f,\label{eq:fundamental}
\end{gather}
and the comparison
\begin{gather}
 \sup_{\sP^1\times\sP^1}|\Phi_F-\log\delta_{\can}|\le 2\sup_{\sP^1}|g_F|<\infty.\label{eq:comparison}
\end{gather}
The $F$-kernel $\Phi_F$ is upper semicontinuous,
and for each Radon measure $\mu$ on $\sP^1$,
introduces the $F$-potential $U_{F,\mu}$ on $\sP^1$ and 
the $F$-energy $I_F(\mu)$ as
\begin{gather*}
U_{F,\mu}(z):=\int_{\sP^1}\Phi_F(z,w)\rd\mu(w),\\
I_F(\mu):=\int_{\sP^1}U_{F,\mu}\rd\mu=\int_{\sP^1\times\sP^1}\Phi_F\rd(\mu\times\mu)
\end{gather*}
(if exists). From (\ref{eq:fundamental}) and the Fubini theorem, 
\begin{gather}
 \Delta U_{F,\mu}=\mu-\mu(\sP^1)\cdot\mu_f.\label{eq:quasipotential}
\end{gather}

A probability Radon measure $\mu$ is called an $F$-equilibrium measure
{\itshape on} $\sP^1$ if $\mu$ maximizes the $F$-energy
among probability Radon measures on $\sP^1$ in that
\begin{gather*}
 I_F(\mu)=V_F:=\sup\{I_F(\nu):\nu\text{ is a probability Radon measure on }\sP^1\},
\end{gather*}
and $V_F$ is called the $F$-equilibrium energy of $\sP^1$.
The comparison (\ref{eq:comparison}) implies that $V_F>-\infty$,
from which there is the unique $F$-equilibrium measure on $\sP^1$
and for every probability Radon measure $\mu$ on $\sP^1$, 
\begin{gather}
 \inf_{z\in\sP^1} U_{F,\mu}(z)\le V_F\le\sup_{z\in\sP^1} U_{F,\mu}(z).\label{eq:infsup}
\end{gather}
For non-archimedean $K$, see also the continuity of $U_{F,\mu}$ 
(\cite[Lemma 5.24]{BR10}, \cite[\S 2.4]{FR09}) and the property of $\Delta$
(\cite[Proposition 8.66]{BR10}, \cite[\S2.4]{FR09}). A characterization of
$\mu_f$ as the unique solution of a Gauss variational problem is

\begin{lemma}\label{th:Frostman}
 The equilibrium measure $\mu_f$ of $f$ is the unique $F$-equilibrium measure on $\sP^1$, 
 that is, 
\begin{gather*}
 \int_{\sP^1\times\sP^1}\Phi_F\rd(\mu_f\times\mu_f)=V_F;
\end{gather*}
indeed, $U_{F,\mu_f}\equiv V_F$.
\end{lemma}

\begin{proof}
 From (\ref{eq:quasipotential}), $\Delta U_{F,\mu_f}=0$, so $U_{F,\mu_f}$ is constant on $\sP^1$.
 From (\ref{eq:infsup}), indeed $U_{F,\mu_f}\equiv V_F$, and $I_F(\mu_f)=V_F$.
\end{proof}

Let us introduce the more canonical $f$-kernel
\begin{gather}
 \Phi_f(z,w):=\Phi_F(z,w)-V_F\label{eq:Arakerov}
\end{gather}
on $\sP^1$, which is independent of choices of $F$ from (\ref{eq:ambiguitygreen}). 
For each Radon measure $\mu$ on $\sP^1$, its $f$-potential is
\begin{gather*}
 U_\mu:=\int_{\sP^1}\Phi_f(\cdot,w)\rd\mu(w)=U_{F,\mu}-\mu(\sP^1)V_F
\end{gather*}
on $\sP^1$. We note that
\begin{gather*}
 \Delta U_{\mu}=\Delta U_{F,\mu}=\mu-\mu(\sP^1)\cdot\mu_f,\\
 U_{\mu_f}=U_{F,\mu_f}-V_f\equiv 0
\end{gather*}
and that the dynamical Green function of $f$
\begin{gather}
 g_f(z):=g_F(z)+\frac{1}{2}V_F=\frac{1}{2}(\log\delta_{\can}(z,z)-\Phi_f(z,z)) \label{eq:comparisoncanonical}
\end{gather}
on $\sP^1$ is independent of choices of $F$.

Recall the definition of $G^F$ in (\ref{eq:escaping}). 
The {\itshape bifurcation potential} of $f$ is the constant 
\begin{gather}
 B(f):=\sum_{j=1}^{2d-2}(G^F(C^F_j)+V_F),\label{eq:bif}
\end{gather}
where $(C_j^F)\subset K^2\setminus\{0\}$
is chosen as $\det DF=\prod_{j=1}^{2d-2}(p\wedge C^F_j)$
(the Jacobian determinant of $F$), and
is independent of choices of both $F$ and $(C_j^F)$. 

For the homogeneous resultant $\Res P$
of homogeneous polynomial endomorphism $P$ on $K^2$,
see \cite[\S6]{DeMarco03}, \cite[\S2.4]{SilvermanDynamics}. 
We only mention that $P$ is non-degenerate if and only if $\Res P=0$.
For each lift $F$ of $f$, the energy formula
\begin{gather*}
 V_F=-\frac{1}{d(d-1)}\log|\Res F|
\end{gather*}
was established in \cite[Theorem 1.5]{DeMarco03} for archimedean $K$,
and in \cite[\S10.2]{BR10} generalized to non-archimedean $K$
(for a simple computation, see \cite[Appendix]{OS11}). 

\begin{lemma}\label{th:natural}
For each linear fractional isometry $h$ on $\bP^1$
under $[\cdot,\cdot]$, put $f_h:=h^{-1}\circ f\circ h$.
Then $\mu_{f_h}=h^*\mu_f$. For each $a\in\bP^1$, let us identify $h^*(a)$ with $h^{-1}(a)$.
Then $f_h^*(h^*(a))=h^*(f^*(a))$, $\eta_{h^*(a),k}(f_h)=\eta_{a,k}(f)$ 
and $D_{h^*(a),k}(f_h)=D_{a,k}(f)$.
For every $(z,w)\in\bP^1\times\bP^1$, 
\begin{gather*}
 \Phi_{f_h}(z,w)=\Phi_f(h(z),h(w)). 
\end{gather*}
For every $a\in\bP^1$ and every $k\in\bN$, 
\begin{gather*}
 \cE_{f_h}(k,h^*(a))=\cE_f(k,a). 
\end{gather*}
For each $C^1$-test function $\phi$ on $\sP^1$, so is $h^*\phi$, and it holds that
$\Lip(h^*\phi)=\Lip(\phi)$ and that $\langle h^*\phi,h^*\phi\rangle=\langle\phi,\phi\rangle$. 
\end{lemma}

\begin{proof}
This is clear except for the displayed two equalities.
Let $F$ be a lift of $f$, and $H$ a lift of $h$ on $K^2$ normalized as $|\det H|=1$. 
Then $F_H:=H^{-1}\circ F\circ H$ is a lift of $f_h$.
Under the normalization,
$H$ also preserves $|\cdot|$ on $K^2$, so $g_{F_H}=g_F\circ h$ on $\sP^1$.
Hence for every $(z,w)\in\bP^1\times\bP^1$,
\begin{align*}
 \Phi_{F_H}(z,w)=&\log[z,w]-g_{F_H}(z)-g_{F_H}(w)\\
=&\log[h(z),h(w)]-g_F(h(z))-g_F(h(w))
=\Phi_F(h(z),h(w)). 
\end{align*}
From the formula on $\Res$ in \cite[Proposition 6.1]{DeMarco03} and
\cite[Exercises 2.12]{SilvermanDynamics},
\begin{multline*}
 \Res(F_H)=\Res(H^{-1}\circ F\circ H)=(\det H^{-1})^d\Res(F\circ H)^{1^2}\\
=(\det H)^{-d}((\Res F)^1(\det H)^{d^2})=(\det H)^d\Res F,
\end{multline*}
and since $|\det H|=1$, $|\Res(F_H)|=|\Res F|$, so $V_{F_H}=V_F$.
Hence for every $(z,w)\in\bP^1\times\bP^1$,
$\Phi_{f_h}(z,w)=\Phi_f(h(z),h(w))$. Finally, for every $a\in\bP^1$ and every $k\in\bN$, 
\begin{align*}
 &\cE_{f_h}(k,h^*(a))\\
=&\frac{1}{d^{2k}}\int_{\sP^1\times\sP^1\setminus\diag}\Phi_{f_h}(z,w)\rd
((f_h^k)^*(h^*(a))\times(f_h^k)^*(h^*(a)))(z,w)\\
=&\frac{1}{d^{2k}}\int_{\sP^1\times\sP^1\setminus\diag}\Phi_f(h(z),h(w))\rd
(h^*(f^k)^*(a)\times h^*(f^k)^*(a))(z,w)=\cE_f(k,a).
\end{align*}

Now the proof is complete.
\end{proof}

Recall the definition \eqref{eq:FRL} of dynamical Favre and Rivera-Letelier bilinear form
$(\mu,\mu')_f$ for Radon measures $\mu,\mu'$ on $\sP^1$.
A classical notion of asymptotic Fekete configuration 
on a compact subset $C$ in $\sP^1$ (see \cite{Fekete30}, \cite{Fekete33})
extends to sequences of positive measures whose supports consist of finitely
many points in $C$. Here, we are only interested in the case of $C=\sP^1$:

\begin{definition}\label{th:fekete}
 A sequence $(\nu_k)$ of positive measures whose supports consist of finitely
 many points on $\sP^1$ is 
 $f$-asymptotically Fekete (or an $f$-asymptotic Fekete configuration) 
 on $\sP^1$ if as $k\to\infty$,
 $\nu_k(\sP^1)\nearrow\infty$,
 $(\nu_k\times\nu_k)(\diag_{\bP^1})=o(\nu_k(\sP^1)^2)$ and
 \begin{gather*}
  \frac{(\nu_k,\nu_k)_f}{\nu_k(\sP^1)^2}\to 0.
 \end{gather*}
\end{definition}

\begin{remark}\label{th:exceptionalfekete}
 For every $a\in\sP^1$ and every $k\in\bN$, $((f^k)^*(a))(\sP^1)=d^k$.
 If $a\in\sH^1$, then $f^{-k}(a)\subset\sH^1$, so 
 $D_{a,k}=0$. If $a\in\bP^1$, then $D_{a,k}\le d^k\eta_{a,k}=o(d^{2k})$
 by (\ref{eq:non-exceptional}). Hence for each $a\in\sP^1\setminus E(f)$, 
 it holds that $((f^k)^*(a))$ is $f$-asymptotically Fekete on $\sP^1$ if and only if
\begin{gather*}
\cE_f(k,a)
=\frac{1}{d^{2k}}\int_{\sP^1\times\sP^1\setminus\diag_{\bP^1}}\Phi_f\rd((f^k)^*(a)\times(f^k)^*(a))\to 0
\end{gather*}
 as $k\to\infty$. For every $a\in E(f)$, $((f^k)^*(a))$ is not $f$-asymptotically Fekete on $\sP^1$
 since for each $k\in\bN$, 
 $D_{a,k}=(((f^k)^*(a))(\sP^1))^2$, so the second condition does not hold.
\end{remark}

\section{Proof of Theorems \ref{th:algebraic}, \ref{th:characterization}, \ref{th:order} and \ref{th:derived}}
\label{eq:proof}
Let $f$ be a rational function on $\bP^1=\bP^1(K)$ of degree $d>1$. 

\begin{lemma}[{cf.\ \cite[Theorem 10.18]{BR10}, \cite[Lemma 1.2]{Sodin92}}]\label{th:Riesz}
 For every $a\in\sP^1$ and every $z\in\sP^1$,
 \begin{gather}
  \Phi_f(f(z),a)=U_{f^*(a)}(z).\label{eq:Riesz}
 \end{gather}
\end{lemma}

\begin{proof}
 For every $a\in\sP^1$, from $f^*\mu_f=d\cdot\mu_f$ and $(f^*(a))(\sP^1)=d$, 
 \begin{gather*}
 \Delta\Phi_f(f(\cdot),a)
=f^*(\Delta\Phi_f(\cdot,a))
=f^*((a)-\mu_f)
=f^*(a)-d\cdot\mu_f
=\Delta U_{f^*(a)}
 \end{gather*}
 on $\sP^1$. Hence $\Phi_f(f(\cdot),a)-U_{f^*(a)}$ is 
 harmonic, so constant, on $\sP^1$.
 Let us integrate this constant function in
 $\rd\mu_f(\cdot)$ over $\sP^1$.
 Then from $f_*\mu_f=\mu_f$ and $U_{\mu_f}\equiv 0$,
 \begin{align*}
 \Phi_f(f(\cdot),a)-U_{f^*(a)}\equiv\int_{\sP^1}(\Phi_F(f(\cdot),a)-U_{f^*(a)})\rd\mu_f=0
 \end{align*}
 on $\sP^1$, which is (\ref{eq:Riesz}).
\end{proof}

The estimate $(\ref{eq:non-classical})$ in Theorem \ref{th:algebraic} follows from
\begin{lemma}
For every $a\in\sH^1$ and every $k\in\bN$,
$\cE_f(k,a)=\Phi_f(a,a)/d^k$.
Moreover, $\sup_{a\in\sP^1}|\Phi_f(a,a)+\rho(a,\cS_{\can})|<\infty$.
\end{lemma}

\begin{proof}
 We assume $k=1$ without loss of generality. Since $f^{-1}(a)\subset\sH^1$,
 integrating (\ref{eq:Riesz}) in $\rd(f^*(a))(z)/d^2$ on $\sP^1$,
 \begin{multline*}
  \frac{1}{d}\Phi_f(a,a)
  =\frac{1}{d^2}\int_{\sP^1\times\sP^1}\Phi_f\rd(f^*(a)\times f^*(a))\\
  =\frac{1}{d^2}\int_{\sP^1\times\sP^1\setminus\diag_{\bP^1}}\Phi_f\rd(f^*(a)\times f^*(a))
  =\cE_f(1,a).
 \end{multline*}
By (\ref{eq:Gromov}), $\rho(a,\cS_{\can})=-\log\delta_{\can}(a,a)$, so
$\sup_{a\in\sP^1}|\Phi_f(a,a)+\rho(a,\cS_{\can})|
=\sup_{a\in\sP^1}|-2g_f(a)|<\infty$.
\end{proof}

From now on, let us assume that $K$ has characteristic $0$.

\begin{definition}
 For every $z\in\bP^1$, put
 \begin{gather}
  c_z(f):=\lim_{\bP^1\ni u\to z}\left\{\Phi_f(f(u),f(z))-(\deg_z f)\Phi_f(u,z)\right\}.\label{eq:defquasiRiesz}
 \end{gather}
\end{definition}

\begin{lemma}\label{th:derivative}
 For every $a\in\bP^1$ and every $z\in f^{-1}(a)$, 
\begin{gather}
 c_z(f)=\int_{\sP^1\setminus\{z\}}\Phi_f(z,w)\rd(f^*(a))(w).\label{eq:quasiRiesz}
\end{gather}
\end{lemma}

\begin{proof}
 For every $a\in\bP^1$ and every $z\in f^{-1}(a)$, from Lemma \ref{th:Riesz},
 for every $u\in\bP^1\setminus\{z\}$,
 \begin{multline*}
  \Phi_f(f(u),f(z))-(\deg_z f)\Phi_f(u,z)
  =\Phi_f(f(u),a)-(\deg_z f)\Phi_f(u,z)\\
  =U_{f^*(a)}(u)-(\deg_z f)\Phi_f(u,z)
  =\int_{\sP^1\setminus\{z\}}\Phi_f(u,w)\rd(f^*(a))(w),
 \end{multline*}
 and take $\lim_{\bP^1\ni u\to z}$ of both sides.
\end{proof}

\begin{lemma}\label{th:error}
 For every $a\in\bP^1$ and every $k\in\bN$,
 \begin{align}
  \cE_f(k,a)=\frac{1}{d^{2k}}\int_{\sP^1}c_z(f^k)\rd((f^k)^*(a))(z).\label{eq:error}
 \end{align}
\end{lemma}

\begin{proof}
 With no loss of generality, we assume that $k=1$.
 Integrating (\ref{eq:quasiRiesz}) in $\rd(f^*(a))(z)/d^2$,
\begin{multline*}
 \int_{\sP^1}c_z(f)\frac{\rd(f^*(a))(z)}{d^2}
=\frac{1}{d^2}\int_{\sP^1}\rd(f^*(a))(z)\int_{\sP^1\setminus\{z\}}\Phi_f(z,w)\rd(f^*(a))(w)\\
=\frac{1}{d^2}\int_{\sP^1\times\sP^1\setminus\diag_{\bP^1}}\Phi_f\rd(f^*(a)\times f^*(a))=\cE_f(1,a).
\end{multline*}
\end{proof}

Put $f^0:=\Id_{\bP^1}$. Recall that $\deg_z f$ is multiplicative in that
for every $z\in\bP^1$ and every $m,n\in\bN\cup\{0\}$,
\begin{gather*}
 \deg_z(f^{m+n})=\deg_{f^n(z)}(f^m)\deg_z(f^n).
\end{gather*}

\begin{lemma}\label{th:chain}
 For every $k\in\bN$ and every $z\in\bP^1$,
 \begin{gather}
  c_z(f^k)=\sum_{j=1}^k\deg_{f^j(z)}(f^{k-j})c_{f^{j-1}(z)}(f).\label{eq:chain}
 \end{gather}
\end{lemma}

\begin{proof}
This is clear if $k=1$. Suppose that $k\ge 2$. For every $z\in\bP^1$ and every $j\in\bN$,
\begin{multline*}
c_z(f^j)\\
=\lim_{\bP^1\ni u\to z}
\bigl\{[\Phi_f(f(f^{j-1}(u)),f(f^{j-1}(z)))
-\deg_{f^{j-1}(z)}(f)\Phi_f(f^{j-1}(u),f^{j-1}(z))]\\
+\deg_{f^{j-1}(z)}(f)[\Phi_f(f^{j-1}(u),f^{j-1}(z))-\deg_z(f^{j-1})\Phi_f(u,z)]\bigr\}\\
=c_{f^{j-1}(z)}(f)+\deg_{f^{j-1}(z)}(f)c_z(f^{j-1}),
\end{multline*}
so
$c_z(f^j)/\deg_z(f^j)-c_z(f^{j-1})/\deg_z(f^{j-1})
=c_{f^{j-1}(z)}(f)/\deg_z(f^j)$.
Taking the sum of this over $j=2,\dots,k$,
\begin{gather*}
 \frac{c_z(f^k)}{\deg_z(f^k)}=\sum_{j=1}^k\frac{c_{f^{j-1}(z)}(f)}{\deg_z(f^j)},
\end{gather*}
which is equivalent to \eqref{eq:chain}.
\end{proof}

\begin{lemma}\label{th:wronskian}
For every $z\in\bP^1\setminus C(f)$,
\begin{gather}
 c_z(f)=-\log|d|+B(f)+\sum_{c\in C(f)}\Phi_f(z,c).\label{eq:derivativeNevan}
\end{gather}
Here the sum takes into account the multiplicity $(\deg_c f-1)$ 
of each $c\in C(f)$. 
\end{lemma}

\begin{proof}
The chordal derivative of $f$ is defined by
\begin{gather*}
  f^\#(z):=\lim_{\bP^1\ni u\to z}[f(u),f(z)]/[u,z]
\end{gather*} 
on $\bP^1$. For every $z\in\bP^1\setminus C(f)$, i.e., $\deg_z f=1$,
 \begin{multline*}
  c_z(f)=\lim_{\bP^1\ni u\to z}\left\{\Phi_f(f(u),f(z))-\Phi_f(u,z)\right\}=
\log f^\#(z)-2g_f(f(z))+2g_f(z).\label{eq:derivative}
 \end{multline*}
Let $F$ be a lift $F$ of $f$. For every $z\in\bP^1$,
by a direct computation involving Euler's identity, 
\begin{gather*}
 f^\#(z)=\frac{1}{|d|}\frac{|p|^2}{|F(p)|^2}|\det DF(p)|
\end{gather*}
if $p\in\pi^{-1}(z)$ (cf.\ \cite[Theorem 4.3]{Jonsson98}). 
Choose $(C_j^F)\subset K^2\setminus\{0\}$ such that
$\det DF(p)=\prod_{j=1}^{2d-2}(p\wedge C^F_j)$.
Then $\pi(C^F_j)$ $(j=1,\ldots,2d-2)$ range over $C(f)$. 
From (\ref{eq:escaping}) and $G^F\circ F=dG^F$,
for every $z\in\bP^1\setminus C(f)$ and $p\in\pi^{-1}(z)$,
\begin{gather*}
 c_z(f)=-\log|d|+\sum_{j=1}^{2d-2}(\log|p\wedge C_j^F|-G^F(p)).
\end{gather*}
For every $(z,w)\in\bP^1\times\bP^1$,
\begin{gather*}
 \Phi_f(z,w)=\log|p\wedge q|-G^F(p)-G^F(q)-V_F
\end{gather*}
if $p\in\pi^{-1}(z)$ and $q\in\pi^{-1}(w)$. This implies that
$\sum_{j=1}^{2d-2}(\log|p\wedge C_j^F|-G^F(p))=\sum_{c\in C(f)}\Phi_f(z,c)+B(f)$.
Now the proof is complete.
\end{proof}

\begin{lemma}\label{th:noncritical}
There is $C_1>0$ such that for every $a\in\bP^1$ and every $k\in\bN$,
 \begin{gather*}
  \left|\cE_f(k,a)-\frac{1}{d^k}\sum_{j=1}^k\sum_{c\in C(f)}
  \int_{\sP^1\setminus\{c\}}\deg_{w}(f^j)\Phi_f(w,c)\rd\frac{(f^j)^*(a)}{d^j}(w)\right|
  \le \frac{C_1}{d^k}\sum_{j=1}^k\frac{D_{a,j}}{d^j}.
 \end{gather*}
The sum over $C(f)$ takes into account the multiplicity $(\deg_c f-1)$ of each $c\in C(f)$. 
\end{lemma}

\begin{proof}
Substituting \eqref{eq:chain} in \eqref{eq:error},
for every $a\in\bP^1$ and every $k\in\bN$, 
 \begin{align*}
 \cE_f(k,a)
=&\frac{1}{d^{2k}}\int_{\sP^1}c_z(f^k)\rd((f^k)^*(a))(z)\\
 =&\frac{1}{d^{2k}}\sum_{j=1}^k\int_{\sP^1}\deg_{f\circ f^{j-1}(z)}(f^{k-j}) c_{f^{j-1}(z)}(f)\rd((f^k)^*(a))(z)\\
 =&\frac{1}{d^k}\sum_{j=1}^k\int_{\sP^1}d^{-k+j-1}\deg_{f(w)}(f^{k-j})c_w(f)\cdot\rd((f^{k-j+1})^*(a))(w)\\
 =&\frac{1}{d^k}\sum_{j=1}^k\int_{\sP^1}d^{-j}\deg_{f(w)}(f^{j-1})c_{w}(f)\rd((f^j)^*(a))(w)\\
 =&\frac{1}{d^k}\sum_{j=1}^k\int_{\sP^1}\frac{\deg_w(f^j)}{\deg_w f}c_{w}(f)\rd\frac{(f^j)^*(a)}{d^j}(w).
 \end{align*}
Hence using Lemma \ref{th:wronskian}, 
\begin{multline*} 
\left|\cE_f(k,a)
-\frac{1}{d^k}\sum_{j=1}^k\sum_{c\in C(f)}
  \int_{\sP^1\setminus\{c\}}\deg_{w}(f^j)\Phi_f(w,c)\rd\frac{(f^j)^*(a)}{d^j}(w)\right|\\
=\Bigg|\frac{1}{d^k}\sum_{j=1}^k
\Biggl(
\int_{C(f)}\deg_w(f^j)\frac{c_w(f)}{\deg_w f}\rd\frac{(f^j)^*(a)}{d^j}(w)+\\
+\int_{\sP^1\setminus C(f)}\deg_{w}(f^j)\left(c_{w}(f)-\sum_{c\in C(f)}\Phi_f(w,c)\right)\rd\frac{(f^j)^*(a)}{d^j}(w)\\
-\sum_{c\in C(f)}
  \int_{C(f)\setminus\{c\}}\deg_{w}(f^j)\Phi_f(w,c)\rd\frac{(f^j)^*(a)}{d^j}(w)\Biggr)\Biggr|
\le \frac{1}{d^k}\sum_{j=1}^kC_1\frac{D_{a,j}}{d^j},
\end{multline*}
where we set
$C_1:=\max_{c\in C(f)}|c_c(f)|+(2d-2)\max_{C(f)\times C(f)\setminus\diag_{\bP^1}}|\Phi_f|+(|\log|d||+|B(f)|)<\infty$.
\end{proof}

\begin{lemma}\label{th:lower}
There is $C_2>0$ such that 
for every $a\in\bP^1$ and every $k\in\bN$,
\begin{multline*}
\frac{1}{d^k}\sum_{j=1}^k\left(\eta_{a,j}\sum_{c\in C(f)\setminus f^{-j}(a)}\frac{1}{d^j}\Phi_f(f^j(c),a)\right)
-\frac{C_2}{d^k}\sum_{j=1}^k\eta_{a,j}-\frac{C_{f,a}}{d^k}\\
\le\frac{1}{d^k}\sum_{j=1}^k\sum_{c\in C(f)}
 \int_{\sP^1\setminus\{c\}}\deg_{w}(f^j)\Phi_f(w,c)\rd\frac{(f^j)^*(a)}{d^j}(w)\\
\le
\frac{1}{d^k}\sum_{j=1}^k\left(\sum_{c\in C(f)\setminus f^{-j}(a)}\frac{1}{d^j}\Phi_f(f^j(c),a)\right)
+\frac{C_2}{d^k}\sum_{j=1}^k\eta_{a,j}+\frac{C_{f,a}}{d^k}.
\end{multline*}
Here the extra constant $C_{f,a}\ge 0$ is independent of $k$, and
even vanishes if $a\in\bP^1\setminus\CO(f)_{\wan}$.
The sums over $C(f)$ take into account the multiplicity $(\deg_c f-1)$ of each $c\in C(f)$. 
\end{lemma}

\begin{proof}
For each $c\in C(f)$, 
\begin{gather}
 \Phi_f(\cdot,c)=\left(\Phi_f(\cdot,c)-\max\{0,\sup_{w\in\bP^1}\Phi_f(w,c)\}\right)
+\max\{0,\sup_{w\in\bP^1}\Phi_f(w,c)\},\label{eq:decomp} 
\end{gather}
where the first and second terms of the right hand side are $\le 0$ and $\ge 0$,
respectively. Set
\begin{gather*}
 C':=(2d-2)\max_{c\in C(f)}\max\left\{0,\sup_{w\in\bP^1}\Phi_f(w,c)\right\}\le
 2(2d-2)\sup_{\bP^1}|g_f|<\infty.
\end{gather*}
Then for every $a\in\bP^1$ and every $j\in\bN$, using the decomposition \eqref{eq:decomp},
\begin{multline*}
\eta_{a,j}\cdot
\sum_{c\in C(f)}\int_{\sP^1\setminus\{c\}}\Phi_f(w,c)\rd\frac{(f^j)^*(a)}{d^j}(w)
-C'\eta_{a,j}\\
\le\sum_{c\in C(f)}
 \int_{\sP^1\setminus\{c\}}\deg_{w}(f^j)\Phi_f(w,c)\rd\frac{(f^j)^*(a)}{d^j}(w)\\
\le\sum_{c\in C(f)}\int_{\sP^1\setminus\{c\}}\Phi_f(w,c)\rd\frac{(f^j)^*(a)}{d^j}(w)
+C'd^{-j}D_{a,j},
\end{multline*}
and using Lemma \ref{th:derivative} for $f^j$,
\begin{multline*}
\eta_{a,j}\left(
\sum_{c\in C(f)\setminus f^{-j}(a)}U_{(f^j)^*(a)/d^j}(c)
+\sum_{c\in C(f)\cap f^{-j}(a)}\frac{1}{d^j}c_c(f^j)\right)
-C'\eta_{a,j}\\
\le\sum_{c\in C(f)}
 \int_{\sP^1\setminus\{c\}}\deg_{w}(f^j)\Phi_f(w,c)\rd\frac{(f^j)^*(a)}{d^j}(w)\\
\le
\sum_{c\in C(f)\setminus f^{-j}(a)}U_{(f^j)^*(a)/d^j}(c)
+\sum_{c\in C(f)\cap f^{-j}(a)}\frac{1}{d^j}c_c(f^j)+C'd^{-j}D_{a,j}.
\end{multline*}
If $c\in (C(f)\setminus C(f)_{\wan})\cap f^{-j}(a)$, then using Lemma \ref{th:chain}, 
\begin{gather*}
\left|\frac{1}{d^j}c_c(f^j)\right|
=\frac{1}{d^j}\sum_{\ell=1}^j\deg_{f^\ell(z)}(f^{j-\ell})|c_{f^{\ell-1}(z)}(f)|
=\frac{1}{d^j}\sum_{\ell=1}^jC''d^{j-\ell}\le C''\frac{1}{d-1},
\end{gather*}
where we set
$C'':=\max_{c\in C(f)\setminus C(f)_{\wan}}\sup_{\ell\in\bN\cup\{0\}}|c_{f^\ell(c)}(f)|<\infty$.
This is finite since any $c\in C(f)\setminus C(f)_{\wan}$ is preperiodic under $f$.
Under the convention that $\sup\emptyset=0$, put
\begin{gather*}
 k_a:=\sup\{j\in\bN;C(f)_{\wan}\cap f^{-j}(a)\neq\emptyset\}<\infty.
\end{gather*}
This is finite since any $c\in C(f)_{\wan}$ is wandering, and
vanishes if and only if $a\in\bP^1\setminus\CO(f)_{\wan}$. 
Under the convention that $\sum_{\emptyset}=0$, put
\begin{gather*}
 C_{f,a}(j):=\sum_{c\in C(f)_{\wan}\cap f^{-j}(a)}|c_c(f^j)|,
\end{gather*}
which vanishes for any $j>k_a$. From these estimates,
\begin{multline}
\eta_{a,j}\left(\sum_{c\in C(f)\setminus f^{-j}(a)}U_{(f^j)^*(a)/d^j}(c)-2C''-C_{f,a}(j)\right)
-C'\eta_{a,j}\\
\le\sum_{c\in C(f)}
 \int_{\sP^1\setminus\{c\}}\deg_{w}(f^j)\Phi_f(w,c)\rd\frac{(f^j)^*(a)}{d^j}(w)\\
(\le\left(\sum_{c\in C(f)\setminus f^{-j}(a)}U_{(f^j)^*(a)/d^j}(c)
+2C''+C_{f,a}(j)\right)+C'd^{-j}D_{a,j})\\
\le\left(\sum_{c\in C(f)\setminus f^{-j}(a)}U_{(f^j)^*(a)/d^j}(c)\right)
+\left(2C''+C_{f,a}(j)+C'\right)\eta_{a,j}.\label{eq:almost}
\end{multline}
By Lemma \ref{th:Riesz}, $U_{(f^j)^*(a)/d^j}(c)=\Phi_f(f^j(c),a)/d^j$.
Set $C_2:=C'+2C''$. Under the convention that $\sum_{j=1}^0=0$, put
\begin{gather*}
C_{f,a}:=\sum_{j=1}^{k_a}C_{f,a}(j)\eta_{a,j},
\end{gather*} 
which vanishes if $a\in\bP^1\setminus\CO(f)_{\wan}$.
For every $a\in\bP^1$ and every $k\in\bN$,
$\sum_{j=1}^kC_{f,a}(j)\eta_{a,j}\le C_{f,a}$.
Take the sum of \eqref{eq:almost} over $j=1,\ldots,k$ 
and divide this by $d^k$. Now the proof is complete.
\end{proof}

Set $C_f:=C_1+C_2+(2d-2)\sup_{\bP^1}|2g_f|$. 
Then Lemmas \ref{th:noncritical} and \ref{th:lower} (with (\ref{eq:comparisoncanonical})
and $D_{a,j}\le d^j\eta_{a,j}$)
completes the proof of \eqref{eq:algebraic} in Theorem \ref{th:algebraic}.

\begin{remark}
 We can set
 \begin{multline*}
  C_f=|B(f)|+\max_{c\in C(f)}|c_c(f)|
 +2\max_{c\in C(f)\setminus C(f)_{\wan}}\sup_{\ell\in\bN\cup\{0\}}|c_{f^\ell(c)}(f)|\\
 -(2d-2)\max_{(c,c')\in C(f)\times C(f)\setminus\diag_{\bP^1}}\log[c,c']+|\log|d||+(8d-8)\sup_{\bP^1}|g_f|.
 \end{multline*}
\end{remark}

\begin{proof}[Proof of Theorem $\ref{th:characterization}$]
First of all, we show 
\begin{gather}
 E_{\wan}(f)=E_{\wan}(f)':=\bigcup_{c\in C(f)_{\wan}}
 \bigcap_{N\in\bN}\bigcup_{j\ge N}B[f^j(c),\exp(-C_0d^j)],\label{eq:accumulateagree}\\
 E_{\wan}(f)'\cap F(f)\subset\bigcup_{c\in C(f)_{\wan}\cap\PS(f)}\bigcap_{N\in\bN}\bigcup_{j\ge N}B[f^j(c),\exp(-C_0d^j)],\label{eq:accumulateFatou}\\
 E_{\wan}(f)'\cap\cJ(f)\subset\bigcup_{c\in C(f)_{\wan}\cap\cJ(f)}
 \bigcap_{N\in\bN}\bigcup_{j\ge N}B[f^j(c),\exp(-C_0d^j)],\label{eq:accumulateJulia}
\end{gather}
which are independent of Theorem \ref{th:algebraic}. 
Suppose that $a\in F(f)\setminus E(f)$. 
Then for every $c\in C(f)\cap\cJ(f)$, $\inf_{j\in\bN}[f^j(c),a]>0$. 
If $a\in\SAT(f)\setminus E(f)$, then there are $\delta>0$ and $d_a\in (1,d)$ such that 
for every $j\in\bN$ and every $c\in C(f)_{\wan}\setminus f^{-j}(a)$,
$[f^j(c),a]\ge\delta\exp(-d_a^j)$. If $a\in\AT(f)\setminus\SAT(f)$,
then there are $C>0$ and $\lambda\in(0,1]$ such that 
for every $j\in\bN$ and every $c\in C(f)_{\wan}\setminus f^{-j}(a)$,
$[f^j(c),a]\ge C\lambda^j$. Hence
Fact \ref{th:classification} implies the inclusion \eqref{eq:accumulateFatou}.
Next, suppose that $a\in\cJ(f)$.
Then for every $c\in C(f)\cap\cF(f)$ and every $j\in\bN$, by (\ref{eq:Riesz})
and (\ref{eq:comparisoncanonical}), 
\begin{gather*}
 \frac{1}{d^j}\Phi_f(f^j(c),a)=U_{(f^j)^*(a)/d^j}(c)
 \ge \min_{w\in f^{-j}(a)}\Phi_f(c,w)\ge\inf_{w\in\cJ(f)}\Phi_f(c,w)>-\infty,
\end{gather*}
and then $\log[f^j(c),a])\ge -C_0d^j$. Hence 
the inclusion \eqref{eq:accumulateJulia} holds. From these inclusions, 
$E_{\wan}(f)'\subset E_{\wan}(f)$, so the equality \eqref{eq:accumulateagree} holds.

Let us show the inclusion (\ref{eq:qualitative}).
For every $a\in\bP^1$ and every $k\in\bN$, by \eqref{eq:algebraic}, 
\begin{align}
&\notag|\cE_f(k,a)|\\
\le&\left((2d-2)\max_{j\in\{1,\ldots,k\}}
\max_{c\in C(f)\setminus f^{-j}(a)}\frac{1}{d^j}\log\frac{1}{[f^j(c),a]}
+C_f+C_{f,a}\right)\frac{1}{d^k}\sum_{j=1}^k\eta_{a,j}.\label{eq:absolutevalue}
\end{align}
For every $a\in\bP^1\setminus E_{\wan}(f)'$, there is $N=N(a)\in\bN$ such that
for every $j>N$ and every $c\in C(f)\setminus f^{-j}(a)$, 
$\log[f^j(c),a]\ge -C_0d^j$. Then
\begin{gather}
\sup_{j\in\bN}\max_{c\in C(f)\setminus f^{-j}(a)}\frac{1}{d^j}\log\frac{1}{[f^j(c),a]}<\infty.\label{eq:upperestimate}
\end{gather}
Hence for every $a\in\bP^1\setminus(E_{\wan}(f)\cup E(f))=\bP^1\setminus(E_{\wan}(f)'\cup E(f))$,
this and \eqref{eq:absolutevalue} together with
\eqref{eq:non-exceptional} imply that $\lim_{k\to\infty}\cE_f(k,a)=0$, so
$a\not\in E_{\Fekete}(f)$.

Let us show the final assertion. Fix a pre-recurrent $c\in C(f)_{\wan}$.
Then since $c$ is pre-recurrent,
for every $N\in\bN$, $\{f^k(c);k\ge N\}\cap\omega(c)$ is dense in $\omega(c)$ under $[\cdot,\cdot]$.
 Hence by the Baire category theorem, 
\begin{gather*}
 B_c:=\bigcap_{N\in\bN}\bigcup_{k\ge N}\left(B[f^k(c),\exp(-d^{3k})]\cap\omega(c)\right)
\end{gather*} 
 is $G_{\delta}$-dense in $\omega(c)$ under $[\cdot,\cdot]$. 
 We have already seen that $E(f)\subset E_{\Fekete}(f)$. 
 For every $a\in B_c\setminus E(f)$,
 since $c$ is wandering, for every $k\in\bN$ large enough,
 $f^k(c)\neq a$. Since $d^{-k}\log[f^k(c),a]\le-d^{2k}$ for infinitely many $k\in\bN$,
 from the upper estimate of $\cE_f(k,a)$ in (\ref{eq:algebraic}) 
 together with (\ref{eq:non-exceptional}),
 $\cE_f(k,a)\le -d^k+o(1)$ as $k\to\infty$. Hence $a\in E_{\Fekete}(f)$. 
\end{proof}

\begin{proof}[Proof of Theorems $\ref{th:order}$ and $\ref{th:derived}$]
 Theorem \ref{th:order} follows by substituting
 (\ref{eq:non-classical}) and (\ref{eq:absolutevalue}) in
 (\ref{eq:cauchyhyp}) and (\ref{eq:cauchy}) in Proposition \ref{th:cauchy}, respectively.

 For every $a\in\bP^1\setminus E_{\wan}(f)=\bP^1\setminus E_{\wan}(f)'$, 
 by \eqref{eq:upperestimate},
\begin{gather*}
  C:=C_{\FRL}\sqrt{(2d-2)\sup_{j\in\bN}\max_{c\in C(f)\setminus f^{-j}(a)}\frac{1}{d^j}\log\frac{1}{[f^j(c),a]}+C_f+C_{f,a}+1}<\infty.
 \end{gather*}
 For every $a_0\in\bP^1\setminus\PC(f)$ and every $r_0\in(0,[a_0,\PC(f)])$,
 \begin{gather*}
  \inf_{a\in B[a_0,r_0]}[a,\PC(f)]
  \ge\begin{cases}
      [a_0,\PC(f)]-r_0 & (\text{if }K\text{ is archimedean})\\
      r_0 & (\text{if }K\text{ is non-archimedean})
     \end{cases},
 \end{gather*} 
 so
 $C':=C_{\FRL}\sqrt{(2d-2)\sup_{a\in B[a_0,r_0]}\log(1/[\PC(f),a])+C_f+1}<\infty$.
 Now both inequalities in Theorem \ref{th:derived}
 follow from (\ref{eq:general}) in Theorem \ref{th:order} and Remark \ref{th:better}.
\end{proof}

\section{Proof of Proposition \ref{th:cauchy}}\label{sec:error}
Let $f$ be a rational function on $\bP^1=\bP^1(K)$ of degree $d>1$.

\begin{lemma}\label{th:reduction}
 For every $a\in\sP^1$ and every $k\in\bN$,
 \begin{gather*}
  \cE_f(k,a)=-\left(\frac{(f^k)^*(a)}{d^k}-\mu_f,\frac{(f^k)^*(a)}{d^k}-\mu_f\right)_f.
 \end{gather*}
\end{lemma}

\begin{proof}
From the bilinearity
and the definition (\ref{eq:errorFekete}) of $\cE_f(k,a)$,
\begin{multline*}
 \left(\frac{(f^k)^*(a)}{d^k}-\mu_f,\frac{(f^k)^*(a)}{d^k}-\mu_f\right)_f\\
=-\cE_f(k,a)-2\left(\frac{(f^k)^*(a)}{d^k},\mu_f\right)_f+(\mu_f,\mu_f)_f.
\end{multline*}
Since $\mu_f$ has no atom in $\bP^1$ and $U_{\mu_f}\equiv 0$ on $\sP^1$,
\begin{multline*}
 -\left(\frac{(f^k)^*(a)}{d^k},\mu_f\right)_f
=\int_{\sH^1}\rd\frac{(f^k)^*(a)}{d^k}(z)\int_{\sP^1}\Phi_f(z,w)\rd\mu_f(w)\\
+\int_{\bP^1}\rd\frac{(f^k)^*(a)}{d^k}(z)\int_{\sP^1\setminus\{z\}}\Phi_f(z,w)\rd\mu_f(w)=0,
\end{multline*}
and similarly, $(\mu_f,\mu_f)_f=0$.
\end{proof}

 The (original) Favre and Rivera-Letelier bilinear form
 $(\mu,\mu')_{\infty}$ (\cite[\S 4.4]{FR06}) was defined by 
 using the (extended) planar logarithmic kernel
 \begin{gather*}
 \log\delta_{\infty}(z,w)=\log\delta_{\can}(z,w)-\log\delta_{\can}(z,\infty)-\log\delta_{\can}(w,\infty)  
 \end{gather*}
 on $\sP^1$ instead of $\Phi_f$ (note that $[z,w]=|z-w|[z,\infty][w,\infty]$ for $z,w\in K$).
For each $a\in\sP^1$ and each $k\in\bN$ such that 
$f^{-k}(a)\subset\sP^1\setminus\{\infty\}$, put
\begin{gather}
 m_k(a):=
\left(\frac{(f^k)^*(a)}{d^k}-\mu_f,\frac{(f^k)^*(a)}{d^k}-\mu_f\right)_{\infty}+\cE_f(k,a).\label{eq:difference}
\end{gather}

\begin{lemma}\label{th:compare}
For every $a\in\sH^1$ and every $k\in\bN$, 
$f^{-k}(a)\subset\sP^1\setminus\{\infty\}$ and $m_k(a)=0$.
Suppose that 
\begin{itemize}
 \item $\infty$ is a repelling periodic point of $f$, or 
 \item $\infty$ is a non-repelling periodic point of $f$ 
       and $K$ is non-archimedean,
\end{itemize}
and put $p:=\min\{j\in\bN;f^j(\infty)=\infty\}$. 
Then for every $r\in(0,1)$ small enough,
there is $C>0$ such that 
for every $a\in\bP^1\setminus\bigcup_{j=0}^{p-1}f^j(B[\infty,r])$
and every $k\in\bN$, $f^{-k}(a)\subset\sP^1\setminus\{\infty\}$ and 
$|m_k(a)|\le Ckd^{-2k}D_{a,k}$.
\end{lemma}

\begin{proof}
Put 
$\phi_{\infty}:=g_f-\log\delta_{\can}(\cdot,\infty)$ on $\sP^1$.
For each $a\in\sP^1$ and each $k\in\bN$, if $f^{-k}(a)\subset\sP^1\setminus\{\infty\}$,
then using Lemma \ref{th:reduction},
\begin{align*}
 &m_k(a)\\
=&-\int_{\sP^1\times\sP^1\setminus\diag_{\bP^1}}
(\phi_{\infty}(z)+\phi_{\infty}(w))
\rd(\left(\frac{(f^k)^*(a)}{d^k}-\mu_f\right)\times\left(\frac{(f^k)^*(a)}{d^k}-\mu_f\right))(z,w)\\
=&-2\int_{\sP^1\times\sP^1\setminus\diag_{\bP^1}}\phi_{\infty}(z)\rd(\left(\frac{(f^k)^*(a)}{d^k}-\mu_f\right)\times\left(\frac{(f^k)^*(a)}{d^k}-\mu_f\right))(z,w)\\
=&2\int_{\diag_{\bP^1}}\phi_{\infty}(z)
\rd(\left(\frac{(f^k)^*(a)}{d^k}-\mu_f\right)\times\left(\frac{(f^k)^*(a)}{d^k}-\mu_f\right))(z,w)\\
=&2\int_{\diag_{\bP^1}}\phi_{\infty}(z)
\rd\left(\frac{(f^k)^*(a)}{d^k}\times\frac{(f^k)^*(a)}{d^k}\right)(z,w).
\end{align*}
Here the third equality follows from
\begin{gather*}
 \int_{\sP^1\times\sP^1}\phi_{\infty}(z)\rd(\left(\frac{(f^k)^*(a)}{d^k}-\mu_f\right)\times\left(\frac{(f^k)^*(a)}{d^k}-\mu_f\right))(z,w)=0,
\end{gather*}
and the final one holds since $\mu_f$ has no atom in $\bP^1$.

In particular, if $a\in\sH^1$, then $f^{-k}(a)\subset\sH^1\subset\sP^1\setminus\{\infty\}$ and 
$m_k(a)=0$.

Suppose that $\infty$ is periodic under $f$. Put 
$p:=\min\{j\in\bN;f^j(\infty)=\infty\}$ and $\lambda:=|(f^p)'(\infty)|$.  
For each $r\in(0,1)$, put
\begin{gather*}
 U(r):=\bigcup_{j=0}^{p-1}f^j(B[\infty,r]).
\end{gather*}

In the case that
$\infty$ is repelling, i.e., $\lambda>1$, for every $r\in(0,1/\sqrt{2})$ small enough,
the single-valued analytic inverse branch $f^{-p}_{\infty}$ of $f^p$ fixing $\infty$ exists on 
$B[\infty,r]$. Note that $|(f^{-p}_{\infty})'(\infty)|=\lambda^{-1}\in(0,1)$.
From the Koebe $1/4$-theorem and its non-archimedean counterpart of
(Koebe) $1$-theorem (cf.\ \cite[Proposition 3.5]{Benedetto03})
applied to this analytic inverse branch of $f^{-p}$,
for every $a\in\bP^1\setminus U(r)$ and every $k\in\bN$,
$[f^{-k}(a),\infty]\ge r/(4\sqrt{2}\lambda)^{k/p+1}>0$ and
\begin{gather*}
 \max_{z\in f^{-k}(a)}|\phi_{\infty}(z)|\le\sup_{\bP^1}|g_f|-\log r+(k/p+1)\log(4\sqrt{2}\lambda).
\end{gather*}
In the case that $\infty$ is non-repelling, i.e., $\lambda\le 1$ 
and that $K$ is also non-archimedean,
for every $r\in(0,1)$ small enough, from the strong triangle inequality
for $[\cdot,\cdot]$, $f^p(B[\infty,r])\subset B[\infty,r]$.
Then for every $a\in\bP^1\setminus U(r)$ and every $k\in\bN$,
$[f^{-k}(a),\infty]\ge r>0$ and
$\max_{z\in f^{-k}(a)}|\phi_{\infty}(z)|\le\sup_{\bP^1}|g_f|-\log r$.

Now the proof is complete.
\end{proof}

The following extends \cite[Th\'eor\`eme 7]{FR06} to general $K$.

\begin{theorem}\label{th:regularization}
 Let $f$ be a rational function on $\bP^1=\bP^1(K)$ of degree $d>1$. 
 Then for every $a\in\sH^1$, every $C^1$-test function $\phi$ on $\sP^1$ and every $k\in\bN$,
\begin{gather}
 \left|\int_{\sP^1}\phi\rd\left(\frac{(f^k)^*(a)}{d^k}-\mu_f\right)\right|
\le\langle\phi,\phi\rangle^{1/2}\sqrt{\left(\frac{(f^k)^*(a)}{d^k}-\mu_f,\frac{(f^k)^*(a)}{d^k}-\mu_f\right)_{\infty}}.\label{eq:FRcauchyhyp}
\end{gather}
 Moreover, there is $C>0$ such that
 for every $a\in\bP^1$, every $C^1$-test function $\phi$ on $\sP^1$ and every $k\in\bN$, 
 if $f^{-k}(a)\subset K=\bP^1\setminus\{\infty\}$, then
\begin{multline}
 \left|\int_{\sP^1}\phi\rd\left(\frac{(f^k)^*(a)}{d^k}-\mu_f\right)\right|
 \le \Lip(\phi)d^{-k}+\\
+\langle\phi,\phi\rangle^{1/2}
\sqrt{\left|\left(\frac{(f^k)^*(a)}{d^k}-\mu_f,\frac{(f^k)^*(a)}{d^k}-\mu_f\right)_{\infty}\right|
+Ckd^{-2k}D_{a,k}}.\label{eq:FRcauchy}
\end{multline}
\end{theorem}

\begin{proof}
For every $a\in\sH^1$ and every $k\in\bN$,
$\Phi_f(f^k(\cdot),a)$ is continuous on $\sP^1$. Hence
by the Cauchy-Schwarz inequality (\cite[(32), (33)]{FR06}), 
\begin{multline*}
 \left|\int_{\sP^1}\phi\rd\left(\frac{(f^k)^*(a)}{d^k}-\mu_f\right)\right|
=\left|\int_{\sP^1}\phi\Delta\frac{\Phi_f(f^k(\cdot),a)}{d^k}\right|\\
\le\langle\phi,\phi\rangle^{1/2}\sqrt{\left(\frac{(f^k)^*(a)}{d^k}-\mu_f,\frac{(f^k)^*(a)}{d^k}-\mu_f\right)_{\infty}},
\end{multline*}
which is \eqref{eq:FRcauchyhyp}.

For each $z\in K$ and each $\epsilon>0$, 
let $[z]_{\epsilon}$ be the ($\epsilon$-)regularization 
of the Dirac measure $(z)$ on $\sP^1$ (for the definition, see \cite[\S 2.6, \S 4.6]{FR06}).
For each $a\in\bP^1$ with $f^{-k}(a)\subset K$, put
\begin{gather*}
 \left[(f^k)^*(a)\right]_{\epsilon}
:=\sum_{z\in f^{-k}(a)}(\deg_z (f^k))\cdot[z]_{\epsilon}.
\end{gather*}
Since the $f$-potential $U_{[(f^k)^*(a)]_{\epsilon}/d^k}$ is also continuous on $\sP^1$ 
(\cite[Lemmes 2.7, 4.8]{FR06}), by the Cauchy-Schwarz inequality,
\begin{multline*}
\left|\int_{\sP^1}\phi\rd\left(\frac{[(f^k)^*(a)]_{\epsilon}}{d^k}-\mu_f\right)\right|
=\left|\int_{\sP^1}\phi\Delta U_{[(f^k)^*(a)]_{\epsilon}/d^k}\right|\\
\le\langle\phi,\phi\rangle^{1/2}\sqrt{\left(\frac{[(f^k)^*(a)]_{\epsilon}}{d^k}-\mu_f,\frac{[(f^k)^*(a)]_{\epsilon}}{d^k}-\mu_f\right)_{\infty}}.
\end{multline*}
Moreover, from the construction of $\epsilon$-regularization of $(z)$,
$\left|\int_{\sP^1}\phi\rd((z)-[z]_{\epsilon})\right|\le\Lip(\phi)\epsilon$
for every $z\in K$. Hence
\begin{gather*}
 \left|\int_{\sP^1}\phi\rd\left(\frac{(f^k)^*(a)}{d^k}
-\frac{[(f^k)^*(a)]_{\epsilon}}{d^k}\right)\right|\le
\Lip(\phi)\epsilon.
\end{gather*}

Recall that $g_{\infty}$ is $\kappa$-H\"older continuous on $\sP^1$ under $\sd$,
where the $\kappa\in(0,1)$ is determined in Fact \ref{th:Holder}.
Let $\eta(\epsilon)$ be the modulus of continuity of $g_{\infty}$ under $\sd$.
In the proof of \cite[Propositions 2.9, 2.10, 4.10, 4.11]{FR06},
it is shown that for every $z\in K$ and every $\epsilon>0$, 
\begin{gather}
\begin{cases}
  |([z]_{\epsilon},\mu_f)_{\infty}-((z),\mu_f)_{\infty}|\le\epsilon+\eta(\epsilon)\quad(z\in K),\label{eq:FR}\\
 ([z]_{\epsilon},[z']_{\epsilon})_{\infty}
 \le\begin{cases}
    ((z),(z'))_{\infty} & (z,z')\in K\times K\setminus\diag_{\bP^1},\\
    C_{\operatorname{abs}}-\log\epsilon & z=z'\in K, 
   \end{cases}
\end{cases}
\end{gather}
where $C_{\operatorname{abs}}\in\bR$ is an absolute constant \cite[p.\ 329]{FR06}. 

From the bilinearity, 
\begin{multline*}
  \left(\frac{[(f^k)^*(a)]_{\epsilon}}{d^k}-\mu_f,\frac{[(f^k)^*(a)]_{\epsilon}}{d^k}-\mu_f\right)_{\infty}\\
 =\frac{1}{d^{2k}}([(f^k)^*(a)]_{\epsilon},[(f^k)^*(a)]_{\epsilon})_{\infty}
 -2\left(\frac{[(f^k)^*(a)]_{\epsilon}}{d^k},\mu_f\right)_{\infty}
 +(\mu_f,\mu_f)_{\infty},
\end{multline*}
and in the right hand side, by (\ref{eq:FR}),
\begin{align*}
 (\text{the first term})
=&\frac{1}{d^{2k}}\sum_{(z,z')\in f^{-k}(a)\times f^{-1}(a)\setminus\diag_{\bP^1}}\deg_z(f^k)\deg_{z'}(f^k)([z]_{\epsilon},[z']_{\epsilon})_{\infty}\\
 &+d^{-2k}\sum_{z\in f^{-k}(a)}(\deg_z(f^k))^2([z]_{\epsilon},[z]_{\epsilon})_{\infty}\\
 \le&\frac{1}{d^{2k}}((f^k)^*(a),(f^k)^*(a))_{\infty}
 +(C_{\operatorname{abs}}-\log\epsilon)d^{-2k}D_{a,k},
\end{align*}
\begin{multline*}
 (\text{the second term})\\
 =-2\sum_{z\in f^{-k}(a)}\frac{\deg_z(f^k)}{d^k}([z]_{\epsilon},\mu_f)_{\infty}
 \le-2\left(\frac{(f^k)^*(a)}{d^k},\mu_f\right)_{\infty}
 +2(\epsilon+\eta(\epsilon)).
\end{multline*}
Set $\epsilon=(d^k)^{-1/\kappa}(<d^{-k}\le d^{-2k}D_{a,k})$ 
and $C:=|C_{\operatorname{abs}}|+(\kappa^{-1}\log d)+4$.
From the bilinearity again, we have
\begin{multline*}
\left(\frac{[(f^k)^*(a)]_{\epsilon}}{d^k}-\mu_f,\frac{[(f^k)^*(a)]_{\epsilon}}{d^k}-\mu_f\right)_{\infty}\\
\le\left(\frac{(f^k)^*(a)}{d^k}-\mu_f,\frac{(f^k)^*(a)}{d^k}-\mu_f\right)_{\infty}
+Ckd^{-2k}D_{a,k}.
\end{multline*}

Now the proof of (\ref{eq:FRcauchy}) is complete.
\end{proof}

Let us complete the proof of Proposition \ref{th:cauchy}.
For every $a\in\sH^1$ and every $k\in\bN$, (\ref{eq:cauchyhyp})
follows from \eqref{eq:FRcauchyhyp}
and (the former half of) Lemma \ref{th:compare}.

Choose distinct cycles $\{z_1,f(z_1),\ldots, f^{p_1-1}(z_1)\}$,
$\{z_2, f(z_2),\ldots, f^{p_2-1}(z_2)\}$ of $f$,
both of which being repelling if $K$ is archimedean, and
where $p_i:=\min\{k\in\bN;f^k(z_i)=z_i\}$ for $i=1,2$. 
If $r_1>0$ and $r_2>0$ are small enough, then
$U_i:=\bigcup_{j=0}^{p_i-1}f^j(B[z_i,r_i])$ for $i\in\{1,2\}$ satisfy
$(\bP^1\setminus U_1)\cup(\bP^1\setminus U_2)=\bP^1$.
Fix $i\in\{1,2\}$. Let $h_i$ be a linear fractional isometry on $\bP^1$
under $[\cdot,\cdot]$ such that $h_i(\infty)=z_i$,
put $f_{h_i}:=h_i^{-1}\circ f\circ h_i$ and identify $h_i^*(a)$ with $h_i^{-1}(a)$.
Then $h_i^*(a)\in\bP^1\setminus\bigcup_{j=0}^{p_i-1}f_{h_i}^j(B[\infty,r_i])$.
Decreasing $r_i>0$ if necessary, from \eqref{eq:FRcauchy},
(the latter half of) Lemma \ref{th:compare} and Lemma \ref{th:natural},
for every $a\in\bP^1\setminus U_i$, 
every $C^1$-test function $\phi$ on $\sP^1$ and every $k\in\bN$,
\begin{align*}
&\left|\int_{\sP^1}\phi\rd\left(\frac{(f^k)^*(a)}{d^k}-\mu_f\right)\right|
=\left|\int_{\sP^1}(h_i^*\phi)\rd\left(\frac{(f_{h_i}^k)^*(h_i^*(a))}{d^k}-\mu_{f_{h_i}}\right)\right|\\
\le& \Lip(h_i^*\phi)d^{-k}+\langle h_i^*\phi,h_i^*\phi\rangle^{1/2}
\sqrt{|\cE_{f_{h_i}}(k,h_i^*(a))|+C_ikd^{-2k}D_{h_i^*(a),k}(f_{h_i})}\\
=& \Lip(\phi)d^{-k}+\langle \phi,\phi\rangle^{1/2}
\sqrt{|\cE_f(k,a)|+C_ikd^{-2k}D_{a,k}(f)}.
\end{align*}
Here the constant $C_i>0$ is independent of $a\in(\bP^1\setminus U_i),\phi$ and $k$.

Set $C_{\FRL}:=2(\max\{C_1,C_2\}+1)$.
Then these estimates for $i\in\{1,2\}$ conclude (\ref{eq:cauchy}).\qed

\begin{acknowledgement}
 The author thanks Professors Matthew Baker, Sebastien Boucksom, Arnaud Ch\'eritat, 
 David Drasin, Charles Favre, Norman Levenberg, Nicolae Mihalache, Juan Rivera-Letelier,
 Robert Rumely and Nessim Sibony for invaluable discussions. The author
 also thanks the referee for very careful scrutiny.
 Most of this work was done during the author's visiting Institut 
 de Math\'ematiques de Jussieu, and the author is grateful to the hospitality.
 The author also thanks ICERM, Brown University, where this was completed.
 This work was partially supported by JSPS Grant-in-Aid for
 Young Scientists (B), 21740096.
\end{acknowledgement}

\def\cprime{$'$}

\end{document}